\documentclass[10pt]{article}

\usepackage{enumerate}
\usepackage{enumitem}
\usepackage{amsmath}
\usepackage{float}
\usepackage{bbm}
\usepackage{comment}
\newcommand{\ifims}[2]{#1} 
\newcommand{\ifAMS}[3]{#3}   
\newcommand{\ifau}[3]{#2}  
\newcommand{\ifbook}[2]{#1}   

\ifbook{

}
{

}

\def\thetitle{Uniform Hanson-Wright type concentration inequalities for unbounded entries via the entropy method}
\def\theruntitle{Uniform Hanson-Wright}

\def\theabstract{
	This paper is devoted to uniform versions of the Hanson-Wright inequality for a random vector $X \in \mathbb{R}^n$ with independent subgaussian components. The core technique of the paper is based on the entropy method combined with truncations of both gradients of {functions} of interest and of the components of $X$ itself. 
	Our results recover, in particular, the classic uniform bound of \cite{Talagrand96} for Rademacher chaoses and the more recent uniform result of \cite{Adam15} which holds under certain rather strong assumptions on the distribution of $X$. 
	We provide several applications of our techniques: we establish a version of the standard Hanson-Wright inequality, which is tighter in some regimes.
	Extending our results we show a version of {the} dimension-free matrix Bernstein inequality that holds for random matrices with a subexponential spectral norm.
	We apply the derived inequality to the problem of covariance estimation with missing observations and prove an almost optimal high probability version of the recent result of \cite{Lounici14}.
	Finally, we show a uniform Hanson-Wright-type inequality in the Ising model under Dobrushin's condition.
	A closely related question was posed by \cite{Marton03}.
}
\def\kwdp{60E15}
\def\kwds{65F05}

\def\thekeywords{concentration inequalities, modified logarithmic Sobolev inequalities, uniform Hanson-Wright inequalities, Rademacher chaos, matrix Bernstein inequality}

\def\thanksa{Financial support from the German Research Foundation (DFG) via the International Research Training Group 1792 ``High Dimensional Nonstationary Time Series'' in Humboldt-Universit\"at zu Berlin is gratefully acknowledged.}
\def\thanksb{Parts of this work were done while the author was a postdoctoral fellow at Technion. Nikita Zhivotovskiy was supported by RSF grant No. 18-11-00132.}
\def\authora{Yegor Klochkov}
\def\authorb{Nikita Zhivotovskiy}
\def\runauthora{Y. Klochkov }
\def\runauthorb{N. Zhivotovskiy }
\def\addressa{
	Humboldt Universit\"at zu Berlin
}
\def\emaila{klochkoy@hu-berlin.de}
\def\addressb{
	Higher School of Economics, Now at Google Research, Brain Team.
}
\def\emailb{nikita.zhivotovskiy@phystech.edu}

\usepackage{amsmath,amssymb,amsthm}
\usepackage{natbib}
\usepackage{epsfig,graphicx}
\usepackage{comment}
\usepackage{color}
\usepackage{srcltx}
\usepackage[mathscr]{eucal}
\usepackage[math]{easyeqn}
\usepackage{etoolbox}
\usepackage{hyperref}
\PassOptionsToPackage{linktocpage}{hyperref}

\definecolor{blue(pigment)}{rgb}{0.2, 0.2, 0.6}
\definecolor{ultramarine}{rgb}{0.07, 0.04, 0.56}
\definecolor{darkspringgreen}{rgb}{0.09, 0.45, 0.27}
\definecolor{hookersgreen}{rgb}{0.0, 0.44, 0.0}
\definecolor{plum(traditional)}{rgb}{0.56, 0.27, 0.52}
\definecolor{purple(html/css)}{rgb}{0.5, 0.0, 0.5}
\definecolor{magenta(dye)}{rgb}{0.79, 0.08, 0.48}

\hypersetup{
	colorlinks,
	linkcolor=hookersgreen,
	linktoc=blue,
	citecolor=ultramarine,
	urlcolor=black,
	filecolor=black
}
\usepackage{nameref}

\ifims{
	\textheight=23cm
	\textwidth=14.8cm
	\topmargin=0pt
	\oddsidemargin=0.5cm
	\evensidemargin=0.5cm
	\linespread{1.0}
	\renewenvironment{abstract}
	{\centerline{\textbf{Abstract}}\bigskip
		\begin{center}
			\begin{minipage}{11cm}
				\begin{small}
					
				}
				{   \end{small}
			\end{minipage}
		\end{center}
		\bigskip
	}

}{ 
}

\numberwithin{equation}{section}
\numberwithin{figure}{section}
\newcounter{example}[section]
\numberwithin{example}{section}
\newcounter{remark}[section]
\numberwithin{remark}{section}
\newtheorem{theorem}{Theorem}[section]

\newtheorem{proposition}[theorem]{Proposition}
\newtheorem{lemma}[theorem]{Lemma}
\newtheorem{corollary}[theorem]{Corollary}

\newtheorem{exmp}[example]{Example}
\newtheorem{rmrk}[remark]{Remark}
\newenvironment{example}{\begin{exmp}\rm}{\end{exmp}}
\newenvironment{remark}{\begin{rmrk}\rm}{\end{rmrk}}
\newtheorem{assumption}{Assumption}

\def\rr{\tilde{\mathbf{r}}}

\begin{document}
	\thispagestyle{empty}
	\ifims{
		\title{\thetitle}
		\ifau{ 
			\author{
				\authora
				\ifdef{\thanksa}{\thanks{\thanksa}}{}
				\\[5.pt]
				\addressa \\
				\texttt{ \emaila}
			}
		}
		{  
			\author{
				\authora
				\ifdef{\thanksa}{\thanks{\thanksa}}{}
				\\[5.pt]
				\addressa \\
				\texttt{ \emaila}
				\and
				\authorb
				\ifdef{\thanksb}{\thanks{\thanksb}}{}
				\\[5.pt]
				\addressb \\
				\texttt{ \emailb}
			}
		}
		{   
			\author{
				\authora
				\ifdef{\thanksa}{\thanks{\thanksa}}{}
				\\[5.pt]
				\addressa \\
				\texttt{ \emaila}
				\and
				\authorb
				\ifdef{\thanksb}{\thanks{\thanksb}}{}
				\\[5.pt]
				\addressb \\
				\texttt{ \emailb}
				\and
				\authorc
				\ifdef{\thanksc}{\thanks{\thanksc}}{}
				\\[5.pt]
				\addressc \\
				\texttt{ \emailc}
			}
		}
		
		\maketitle
		\pagestyle{myheadings}
		\markboth
		{\hfill \textsc{ \small \theruntitle} \hfill}
		{\hfill
			\textsc{ \small
				\ifau{\runauthora}
				{\runauthora and \runauthorb}
				{\runauthora, \runauthorb, and \runauthorc}
			}
			\hfill}
		\begin{abstract}
			\theabstract
		\end{abstract}
		
		\ifAMS
		{\par\noindent\emph{AMS 2010 Subject Classification:} Primary \kwdp. Secondary \kwds}
		{\par\noindent\emph{JEL codes}: \kwdp}
		{}
		
	} 
	{ 
		\begin{frontmatter}
			\title{\thetitle}

			
			\runtitle{\theruntitle}
			
			\ifau{ 
				\begin{aug}
					\author{\authora\ead[label=e1]{\emaila}}
					\address{\addressa \\
						\printead{e1}}
				\end{aug}
				
				\runauthor{\runauthora}
				\affiliation{\affiliationa} }
			{ 
				\begin{aug}
					\author{\authora\ead[label=e1]{\emaila}\thanksref{t21}}
					\and
					\author{\authorb\ead[label=e2]{\emailb}\thanksref{t22}}
					
					\address{\addressa \\
						\printead{e1}}
					\address{\addressb \\
						\printead{e2}}
					\thankstext{t21}{\thanksa}
					\thankstext{t22}{\thanksb}
					\affiliation{\affiliationa, \affiliationb} 
					\runauthor{\runauthora and \runauthorb}
				\end{aug}
			} 
			{ 
				\begin{aug}
					\author{\authora\ead[label=e1]{\emaila}\thanksref{t21}}
					\and
					\author{\authorb\ead[label=e2]{\emailb}\thanksref{t22}}
					\and
					\author{\authorc\ead[label=e3]{\emailc}\thanksref{t23}}
					
					\address{\addressa \\
						\printead{e1}}
					\address{\addressb \\
						\printead{e2}}
					\address{\addressc \\
						\printead{e3}}
					\thankstext{t21}{\thanksa}
					\thankstext{t22}{\thanksb}
					\thankstext{t23}{\thanksc}
					\affiliation{\affiliationa, \affiliationb, \affiliationc} 
					\runauthor{\runauthora, \runauthorb, and \runauthorc}
			\end{aug}}

			\begin{abstract}
				\theabstract
			\end{abstract}

			
			
		\end{frontmatter}
	} 
	
	\newenvironment{myexample}[2][]{\refstepcounter{example}\par\medskip
		\noindent \textbf{Example~\theexample} (#2)\textbf{.}\rmfamily}{\medskip}
	
	\newcounter{exercise}[section]
	\numberwithin{exercise}{section}
	\newtheorem{exrc}[exercise]{Exercise}
	\newenvironment{exercise}{\begin{exrc}\rm}{\end{exrc}}
	
	
	
	\def\betav{\bb{\beta}}
	\def\gammav{{\boldsymbol{\gamma}}}
	\def\deltav{\boldsymbol{\delta}}
	\def\xv{\mathbf{x}}
	\def\av{\mathbf{a}}
	\def\ev{\mathbf{e}}
	\def\uv{\mathbf{u}}
	\def\gv{G}
	\def\tv{\mathbf{t}}
	\def\yv{\mathbf{y}}
	
	\def\entrlq{\entrl_{1}}
	\def\entrlg{\entrl_{2}}

	\def\R{\mathbb{R}}
	\def\E{\mathbb{E}}
	\def\P{\mathbb{P}}
	\def\C{\mathbb{C}}
	\def\X{\mathbb{X}}
	\def\Rplus{\R_{+}}
	\def\S{\mathbb{S}}
	\def\H{\mathbb{H}}
	\def\Splus{\S_{+}}
	\def\N{\mathbb{N}}
	\def\BigO{\mathcal{O}}
	\def\T{\top}
	
	%
	%
	\def\Normal{\mathcal{N}}
	\def\FF{\mathcal{F}}
	\def\SS{\mathcal{S}}
	\def\WW{\mathcal{W}}
	\def\XX{\mathcal{X}}
	\def\AA{\mathcal{A}}
	\def\BB{\mathcal{B}}
	\def\II{\mathcal{I}}
	\def\HH{\mathcal{H}}
	
	\def\dist{\mathsf{d}}
	\def\Dist{\mathsf{D}}
	\def\sign{\text{sign}}
	\def\Arg{\text{Arg}}
	\def\cone{\text{Cone}}
	\def\Cone{\cone}
	\def\conv{\text{Conv}}
	\def\tr{\mathrm{Tr}}
	\def\Tr{\tr}
	\def\Diag{\mathrm{Diag}}
	\def\diag{\mathrm{diag}}
	\def\Off{\mathrm{Off}}
	\def\floor#1{\lfloor #1 \rfloor}
	\def\ceil#1{\lceil #1 \rceil}
	\def\Ind{\boldsymbol{1}}
	\def\Ent{\mathrm{Ent}}
	
	\def\dmin{d_{\min}}
	\def\dmax{d_{\max}}
	\def\do{d^{\circ}}
	
	\def\eps{\varepsilon}
	\def\epsv{{\boldsymbol \eps}}

	\def\gh{{\widehat{g}}}

	%
	%
	\definecolor{light-green}{rgb}{0.9, 1, 0.8}
	\definecolor{dark-green}{rgb}{0.1, 0.5, 0.1}
	\definecolor{light-red}{rgb}{1, 0.9, 0.8}
	%
	%
	\def\ornstart{|\mspace{-8mu}-\mspace{-8mu}}
	\def\orncontinue{\langle\mspace{-3mu}\rangle\mspace{-8mu}-\mspace{-8mu}}
	\def\ornend{|}
	\def\ornlineX{\ornstart\orncontinue\orncontinue\orncontinue\orncontinue\orncontinue\orncontinue\orncontinue\orncontinue\orncontinue\orncontinue\orncontinue\orncontinue\orncontinue\orncontinue\orncontinue\orncontinue\orncontinue\orncontinue\orncontinue\orncontinue\orncontinue\orncontinue\orncontinue\orncontinue\orncontinue\orncontinue\orncontinue\orncontinue\orncontinue\orncontinue\orncontinue\orncontinue\orncontinue\orncontinue\orncontinue\orncontinue\orncontinue\orncontinue\orncontinue\orncontinue\ornend}
	\def\ornline{\[{\ornlineX}\]}
	
	\def\lambdao{\overline{\lambda}}
	
	\newcommand{\vertiii}[1]{{\left\vert\kern-0.25ex\left\vert\kern-0.25ex\left\vert #1 
			\right\vert\kern-0.25ex\right\vert\kern-0.25ex\right\vert}}
	
	\section{Introduction}
	The concentration properties of quadratic forms of random variables is a classic topic in probability. A well-known result is due to Hanson and Wright (we refer to the form of this inequality presented in \cite{Rudelson13}), which claims that if $A$ is an $n \times n$ real matrix and $X = (X_1, \ldots, X_n)$ is a random vector in $\mathbb{R}^n$ with independent centered components satisfying $\max_{i}\|X_i\|_{\psi_2} \le K$ (we will recall the definition of $\| \cdot \|_{\psi_2}$ below), then for all $t \ge 0$
	\begin{equation}
	\label{HansonWright}
	\P(|X^\T AX - \E X^\T AX| \ge t) \le 2\exp\left(-c \min\left\{\frac{t^2}{K^4\|A\|^2_{\text{HS}}}, \frac{t}{K^2\|A\|}\right\}\right),
	\end{equation}
	for some absolute $c > 0$, where $\|A\|_{\text{HS}} = \sqrt{\sum_{i, j} A_{i, j}^2}$ defines the Hilbert-Schmidt norm and $\|A\|$ is {the} operator norm of $A$. An important extension of these results is when instead of just one matrix $A$ we have a family of matrices $\mathcal{A}$ and want to understand the behaviour of random quadratic forms simultaneously for all matrices in the family. As a concrete example we consider an order-2 Rademacher chaos: given a family $\mathcal{A} \subset \R^{n \times n} $ of $n \times n$ real symmetric matrices with zero diagonal, that is for all $A \in \mathcal{A}$ we have $A_{ii} = 0$ for all $i = 1, \ldots, n$, one wants to study the following random variable
	\begin{equation*}
	{Z_{\AA}(\eps)} = \sup\limits_{A \in \mathcal{A}}\sum\limits_{i, j = 1}^n A_{ij}\varepsilon_i\varepsilon_j = \sup\limits_{A \in \mathcal{A}}\varepsilon^{\T}A\varepsilon,
	\end{equation*}
	where $\varepsilon = (\varepsilon_{1}, \ldots, \varepsilon_n)^{\T}$ is a sequence of independent Rademacher signs taking values $\pm 1$ with equal probabilities. In the celebrated paper \cite{Talagrand96} it was shown, in particular, that there is an absolute constant $c > 0$, such that for any $t \ge 0$
	\begin{equation}
	\label{Talagrand}
	\P(|Z_{\AA}(\eps) - \E Z_{\AA}(\eps)| \ge t) \le 2 \exp\left(-c\min\left(\frac{t^2}{(\E\sup\limits_{A \in \mathcal{A}}\|AX\|)^2}, \frac{t}{\sup\limits_{A \in \mathcal{A}}\|A\|}\right)\right).
	\end{equation}
	{Similar inequalities in the Gaussian case follow from the results in \cite{Borell84} and \cite{Arcones93}. Apart from the new techniques that were used to prove \eqref{Talagrand}, the significance of this result is that previously (see, for example, \cite{ledoux2013probability}) similar bounds were one-sided and had a multiplicative constant greater than $1$ before $\E Z_{\AA}(\eps)$. Results with a multiplicative factor not equal to $1$ are usually called \emph{deviation inequalities} in contrast to \emph{concentration bounds} of the form \eqref{Talagrand} that are studied below. A simplified proof of the {upper tail} of \eqref{Talagrand}, that is the upper bound on $\P(Z_{\AA}(\eps) - \E Z_{\AA}(\eps) \ge t)$, appeared later in \cite{Boucheron03}. We will refer to inequalities of this form as ({one-sided}) \emph{concentration inequalities}.
	
	It is worth mentioning in advance that our main results are one-sided concentration inequalities. This is because the entropy method, used extensively in our proofs, is known to have some limitations when applied to prove lower tail inequalities (see the discussions in \cite{Ledoux2001, boucheron2013concentration}). It would be interesting for future work to consider similar bounds for the lower tails. 
	}
	
	Observe that when for every $A \in \mathcal{A}$ the diagonal elements are zero, the corresponding quadratic forms are centered, that is $\E \varepsilon^TA\varepsilon = 0$. In the general situation we will be interested in the analysis of 
	\begin{equation}
	\label{zdef}
	{Z_{\AA}(X)} = \sup\limits_{A \in \mathcal{A}}(X^{\T}AX - \E X^{\T}AX),
	\end{equation}
	for a random vector $X$ taking its values in $\mathbb{R}^n$. The analysis of both the expectation and the concentration/deviation properties of this random variable has appeared recently in many papers. To name several deviation inequalities: \cite{Mendelson14} study $\E Z_{\AA}(X)$ and deviations of $Z_{\AA}(X)$ for classes of positive semidefinite matrices with applications to compressive {sensing}, \cite{Dicker17} prove deviation inequalities for $\sup_{A \in \mathcal{A}}(X^{\T}AX - \E X^{\T}AX)$ and subgaussian vectors $X$ under some extra assumptions.
	Additionally, a recent paper \cite{adamczak2018hanson} studies deviation bounds for \( Z = \| X^{\T} A X - \E X^{\T} A X \| \) with Banach space-valued matrices \( A \) and Gaussian variables, providing upper and lower bounds for the moments. {The deviation inequality for general subgaussian vectors and a single positive semi-definite matrix was obtained in \cite{Hsu12}. Returning to concentration inequalities}, it was shown in \cite{Adam15} that if $X$ satisfies the so-called \emph{concentration property} with constant $K$, that is for every $1$-Lipschitz function $\varphi:\mathbb{R}^n \to \mathbb{R}$  and any $t \ge 0$ we have $\E|\varphi(X)| < \infty$ and
	\begin{equation}
	\label{concproperty}
	\P\left(|\varphi(X) - \E\varphi(X)| \ge t \right) \le 2\exp\left(-t^2/2K^2 \right),
	\end{equation}
	then the following bound, similar to \eqref{Talagrand}, holds for every $t \ge 0$,
	\begin{equation}
	\label{adam}
	\P(|Z_{\AA}(X) - \E Z_{\AA}(X)| \ge t) \le 2 \exp\left(-c\min\left(\frac{t^2}{K^2(\E\sup\limits_{A \in \mathcal{A}}\|AX\|)^2}, \frac{t}{K^2\sup\limits_{A \in \mathcal{A}}\|A\|}\right)\right).
	\end{equation}
	This result has application to covariance estimation and recovers another recent concentration result of \cite{Koltch17}; this is discussed further in Section \ref{application}. The drawback of \eqref{adam} is that the required concentration property places strong restrictions on the distribution of $X$: while it is satisfied by the standard Gaussian distribution as well as by some log-concave distributions (see \cite{Ledoux2001}), it is not known whether the concentration property holds for general subgaussian entries and even in the simplest case of Rademacher random vectors.
	
	In this paper we extend the aforementioned results in two directions. We extend the result of \cite{Boucheron03} for bounded variables by allowing non-zero diagonal values of the matrices and unbounded subgaussian variables \( X_i \). 
	First, let us recall the following definition. For \( \alpha > 0 \) denote the \( \psi_{\alpha} \)-norm of a random variable $Y$ by
	\[
	\|Y\|_{\psi_\alpha} = \inf\left\{t \ge 0: \E \exp\left(\frac{|Y|^\alpha}{t^\alpha}\right) \le 2\right\},
	\]
	which is a proper norm whenever \( \alpha \geq 1 \).
	A random variable $Y$ with $\|Y\|_{\psi_1} < \infty$ is referred to as subexponential and $\|Y\|_{\psi_2} < \infty$ is referred to as subgaussian and the corresponding norm is usually named a subgaussian norm. We also use the $L_p(P)$ norm. For $p \ge 1$ we set $\|Y\|_{L_p} = (\E|Y|^p)^{\frac{1}{p}}$. One of our main contributions is the following upper-tail bound. 
	\begin{theorem}
		\label{mainthm}
		Suppose that the components of \(X =  (X_1, \dots, X_n) \) are independent centered random variables and $\mathcal{A}$ is a finite family of $n \times n$ real symmetric matrices. Denote $M = \bigl\| \max_{i} |X_i| \bigr\|_{\psi_2}$. Then, for any \( t \geq \max\{M \E \sup_{A \in \AA} \|A X \|, M^{2} \sup_{A \in \AA} \| A \|\} \) {we have}
		\begin{equation}\label{hason_wright_unbounded}
		    \P({Z_{\AA}(X) - \E Z_{\AA}(X)} \ge t) \le \exp\left(-c\min\left(\frac{t^2}{M^2(\E\sup\limits_{A \in \mathcal{A}}\|AX\|)^2}, \frac{t}{M^2\sup\limits_{A \in \mathcal{A}}\|A\|}\right)\right),
		\end{equation}
		where \( c > 0 \) is an absolute constant and $Z_{\AA}(X)$ is defined by \eqref{zdef}.
	\end{theorem}
	
	\begin{remark}
		In Theorem \ref{mainthm} and below we assume that all $A \in \mathcal{A}$ {are} symmetric. This was done only for convenience of presentation and in fact, the analysis may be performed for general square matrices. The only difference will be that in many places $A$ should be replaced by $\frac{1}{2}(A + A^T)$.
	\end{remark}
	{\begin{remark}
	    Notice that even though the above result is stated for finite sets \( \AA \), it also holds for arbitrary bounded sets. Indeed, for a bounded set of matrices \( \AA \), since these matrices are finite dimensional we can consider an increasing sequence \( \AA_{1} \subset \AA_{2} \subset \dots ... \subset \AA \) of finite epsilon-nets of \( \AA \) such that the pointwise convergence \( Z_{\AA_k}(X) \rightarrow Z_{\AA}(X) \) holds. This pointwise convergence implies convergence in probability, in particular,
	    \[
	        \lim_{k \rightarrow \infty} \P( Z_{\AA_k}(X) - \E Z_{\AA_k}(X) \geq t ) =  \P( Z_{\AA}(X) - \E Z_{\AA}(X) \geq t ).
	    \]
	    Since for a subset \( \mathcal{A}_k \subset \AA \) the values \( \E \sup_{A \in \AA_k} \| A X \|^{2} \) and \( \sup_{A \in \AA_k} \| A \| \) are not greater than those for the original set \( \AA \), we obtain the bound \eqref{hason_wright_unbounded} for arbitrary bounded sets. 
	    For the sake of simplicity, we only consider finite sets below.
	\end{remark}}
	
	In particular, Theorem \ref{mainthm} recovers the right-tail of the result of Talagrand \eqref{Talagrand} up to absolute constants, since in  this case we obviously have $\bigl\| \max_{i} |\varepsilon_i| \bigr\|_{\psi_2} \lesssim 1$. Furthermore, the result of Theorem \ref{mainthm} works without the assumption used in \cite{Talagrand96} and \cite{Boucheron03} that diagonals of all matrices in $\mathcal{A}$ are zero. Moreover, it is also applicable in some situations when the concentration property \eqref{concproperty} holds: indeed, if $X$ is a standard normal vector in $\mathbb{R}^n$ then it is well known (see \cite{ledoux2013probability}) that $M = \bigl\| \max_{i} |X_i| \bigr\|_{\psi_2} \sim \sqrt{\log n}$. If moreover the identity matrix $I_n \in \mathcal{A}$ then $\E\sup_{A \in \mathcal{A}} \| A X \| \ge \E\|X\| \gtrsim \sqrt{n}$. Therefore, in this case the factor $M$ is only of at most logarithmic order when compared to $\E\sup_{A \in \mathcal{A}} \| A X \|$. 
	
	In the special case that $\mathcal{A}$ consists of just one matrix, our bound recovers the bound that is similar to the original Hanson-Wright inequality. On the one hand, our bound may have an extra logarithmic factor that depends on the dimension $n$. On the other hand, the original term $\max_{i}\|X_i\|_{\psi_2}\|A\|_{\text{HS}}$ is replaced by the better term $\E\|AX\|$. We discuss this phenomenon below. The core of the proof of the Hanson-Wright inequality in \cite{Rudelson13} is based on the decoupling technique which may be used (at least in a straightforward way) to prove the deviation inequality---but not the concentration inequality---for $\sup_{A \in \mathcal{A}}(X^{\T}AX - \E X^{\T}AX)$ in the case that $\mathcal{A}$ consists of more than one matrix. 
	
	A natural question to ask is whether one may improve Theorem \ref{mainthm} and replace $M = \bigl\| \max_{i} |X_i| \bigr\|_{\psi_2}$ by $K = \max_{i}\bigl\| X_i \bigr\|_{\psi_2}$. In Section \ref{application} we discuss that in the deviation version of Theorem \ref{mainthm} this replacement is not possible in some cases. This is quite unexpected in light of the fact that $\bigl\| \max_{i} |X_i| \bigr\|_{\psi_2}$ does not appear in the original Hanson-Wright inequality. Therefore, we believe that the form of our result is close to optimal. We also provide the following extension of Theorem \ref{mainthm} which may be better in some cases. 
	
	\begin{proposition}
		\label{mainthm2}
		Suppose that the components of \(X =  (X_1, \dots, X_n) \) are independent centered random variables. Suppose also that the variables \( X_i \) are distributed symmetrically ($X_i$ has the same distribution as $-X_i$). Let $\mathcal{A}$ be a finite family of $n \times n$ real symmetric matrices. Denote $M = \bigl\| \max_{i} |X_i| \bigr\|_{\psi_2}$ and $K = \max_{i}\bigl\| X_i \bigr\|_{\psi_2}$ and let $\gv$ be a standard Gaussian vector in $\mathbb{R}^n$. Then, for any $ t \ge \max\{MK \E \sup_{A \in \AA} \|A G \|, MK \sup_{A \in \AA} \| A \|\} $ we have
		\[
		\P({Z_{\AA}(X) - \E Z_{\AA}(X)} \ge t) \le \exp\left(-c\min\left(\frac{t^2}{M^2K^2(\E\sup\limits_{A \in \mathcal{A}}\|AG\|)^2}, \frac{t}{MK\sup\limits_{A \in \mathcal{A}}\|A\|}\right)\right),
		\]
		where \( c > 0 \) is an absolute constant and ${Z_{\AA}(X)}$ is defined by \eqref{zdef}.
	\end{proposition}
	\begin{remark}
		Proposition \ref{mainthm2} is closer to the standard Hanson-Wright inequality~\eqref{HansonWright}. Indeed, in the case that $\mathcal{A} = \{A\}$ we have $\E\|AG\| \sim \|A\|_{\text{HS}}$. The difference is that $K^4$ and $K^2$ are replaced by $M^2K^2$ and $MK$ respectively.
	\end{remark}
	We proceed with some {notation} that will be used below. For a non-negative random variable~$Y$, define its \emph{entropy} as
	\[
	\Ent(Y) = \E Y\log Y - \E Y\log \E Y.
	\]
	Instead of the concentration property \eqref{concproperty}, we also discuss the following closely related property:
	\begin{assumption}
		\label{logSobolev}
		We say that a random vector $X$ taking value in $\mathbb{R}^n$ satisfies the logarithmic Sobolev inequality with constant $K > 0$ if for any continuously differentiable function $f: \mathbb{R}^n \to \mathbb{R}$ we have
		\begin{equation}\label{log-sobol_assume}
		\Ent(f^2) \le 2K^2\E\|\nabla f(X)\|^2,
		\end{equation}
		whenever both sides of the inequality are not infinite.
	\end{assumption}
	
	One of the technical contributions of this paper is that we use a similar scheme to prove Theorem \ref{mainthm} and to recover \eqref{adam} under the logarithmic Sobolev Assumption \ref{logSobolev}. The application of logarithmic Sobolev inequalities requires computation of the gradient of the function of interest, that is, in our case, the gradient of $Z_{\AA}(X) = \sup\limits_{A \in \mathcal{A}}(X^TAX - \E X^T A X)$. In the analysis that we present, there is a need to control the behaviour of $\nabla Z_{\AA}(X)$ (or its analogs) and, as in \cite{Boucheron03} and \cite{Adam15}, we use a truncation argument to do so. However, in both cases our proofs make use of the \emph{entropy variational formula} of \cite{boucheron2013concentration}, that states that for random variables $Y, W$ with $\E\exp(W) < \infty$ we have
	\begin{equation}
	\label{varformula}
	\E(W\exp(\lambda Y)) \le \E\exp(\lambda Y)\log(\E\exp(W)) + \Ent(\exp(\lambda Y)).
	\end{equation}
	Doing so allows us to shorten the proofs and avoid some technicalities appearing in previous papers. Finally, to prove Theorem \ref{mainthm} we use a second truncation argument: this argument is based on {the} Hoffman-J\o{}rgensen inequality (see \cite{ledoux2013probability}). We also present two lemmas which are used several times in the text. Both results have short proofs and may be of independent interest.
	\begin{lemma}
		\label{grad_trunc_lem}
		Suppose that for random variables $Z, W$ and any \( \lambda > 0 \) we have
		\begin{equation}
		\label{mod_log-sob_1}
		\Ent(e^{\lambda Z}) \leq \lambda^{2} \E W e^{\lambda Z} \quad\quad \text{and} \quad\quad \P( W > L + \theta t) \leq e^{-t},
		\end{equation}
		where $\theta, L$ are positive constants. Then, the following concentration result holds
		\begin{equation}
		\label{decoupled_bound}
		\P( Z - \E Z > t ) \leq \exp\left( -c \min\left\{\frac{t^2}{L + \theta}, \frac{t}{\sqrt{\theta}}\right\}  \right),
		\end{equation}
		where \( c > 0 \) is an absolute constant. If, moreover, \eqref{mod_log-sob_1} holds for any \( \lambda \le 0\), we have 
		\[
		\P( |Z - \E Z| > t ) \leq {2}\exp\left( -c \min\left\{\frac{t^2}{L + \theta}, \frac{t}{\sqrt{\theta}}\right\} \right).
		\]
	\end{lemma}
	
	The second technical result is a version of the convex concentration inequality of \cite{Talagrand96} which does not require the boundedness of the components of $X$. 
	\begin{lemma}
		\label{convexunbound}
		\label{convex_subgauss:lemmamain}
		Let \( f: \mathbb{R}^n \to \mathbb{R} \) be a convex, \( L \)-Lipschitz function with respect to the Euclidean norm on \( \R^{n} \) and $X = (X_1, \ldots, X_n)$ be a random vector with independent components. Then, for any \( t \ge  CL\left\|\max_{i} |X_i| \right\|_{\psi_2}\) we have
		\[
		\P\left( |f(X) - \E f(X)| > t\right) \le \exp\left(-c\frac{t^2}{L^2 \left\| \max_{i} |X_i| \right\|^2_{\psi_2}}\right),
		\]
		where $c, C > 0$ are absolute constants.
	\end{lemma}

	{Despite generalizing existing results on convex concentration, the result of Lemma \ref{convexunbound} follows easily from the truncation approach combined with the Hoffman-J\o{}rgensen inequality. As another application of this technique we provide a version of the matrix Bernstein inequality that holds for random matrices with subexponential spectral norm. For clarity of presentation, this inequality is first presented in Section \ref{matrixbern}. Finally, the same argument showing that it is not possible to replace $M = \bigl\| \max_{i} |X_i| \bigr\|_{\psi_2}$ by $K = \max_{i}\bigl\| X_i \bigr\|_{\psi_2}$ in Theorem \ref{mainthm} is used to show that the same is not possible in Lemma
	\ref{convexunbound}.}
	
	We sum up the structure of the paper: 
	\begin{itemize}
		\item Section~\ref{application} is devoted to applications and discussions and consists of several parts. At first, we give a simple proof of the uniform bound of \cite{Adam15} under the logarithmic Sobolev assumption. The second paragraph is devoted to improvements of the non-uniform Hanson-Wright inequality \eqref{HansonWright} in the subgaussian regime. Furthermore, we apply our techniques to obtain a uniform concentration result similar to Theorem~\ref{mainthm} in a particular case of non-independent components. We consider the Ising model under Dobrushin's condition, a setting that has been studied recently by \cite{Adam18} and \cite{Gotze18}. The question we study was raised by \cite{Marton03} in a closely related scenario. Finally, we show that it is not possible in general to replace \( \| \max_{i} |X_i| \|_{\psi_{2}} \) with \( \max_{i} \| X_{i} \|_{\psi_2} \) in Theorem~\ref{mainthm} by providing an appropriate counterexample.
		\item In Section \ref{expef} we present our proof of Theorem \ref{mainthm}. While doing so, we prove Lemma \ref{grad_trunc_lem} and Lemma \ref{convex_subgauss:lemmamain}. Finally, we give a proof of Proposition \ref{mainthm2}. 
		\item In Section \ref{matrixbern} we formulate and prove the dimension-free matrix Bernstein inequality that holds for random matrices with subexponential spectral norm. We demonstrate how our Bernstein inequality can be used in the context of covariance estimation for subgaussian observations improving the state-of-the-art result of \cite{Lounici14} for covariance estimation with missing observations.
	\end{itemize}

	\section{Some applications and discussions}\label{application}
	We begin with some {notation} that will be used throughout the paper. For a random vector $X$ taking its values in $\mathbb{R}^n$ let $X_1, \ldots, X_n$ denote its components. {When all components of $X$ are independent let $X_i'$ denote  {an} independent copy of the component $X_i$.} {Throughout the paper $C,c > 0$ are absolute constants that may change from line to line. We write $ a \lesssim b $ if \( a \leq C b \) for some absolute constant \( C > 0\). Moreover, if \( a \lesssim b \) and \( b \lesssim a \) we write \( a \sim b \).}
	
	{Furthermore, for a square matrix \( A \), denote by \( \Diag(A) \) the diagonal matrix that has the same elements on the diagonal as \(A \). The off-diagonal part of \( A\) is defined by \( \Off(A) = A - \Diag(A) \); we define \( \diag(\av) \) as a \( n \times n\) diagonal matrix with diagonal elements \( \av \in \R^{n} \). Finally, for two symmetric (Hermitian) matrices \( A, B \) we write \( A \prec B \) if \( B - A \)  is positive-definite and \( A \preceq B \) if \( B - A \)  is positive-semidefinite. }{In what follows we also use the following equivalent formulations of tail inequalities. Assume that for a random variable $Y$ and some $a, b > 0$ we have that for any $t \ge 1$,
	\[
	    \P(Y > \max(a\sqrt{t}, b t) ) \leq e^{-t} .
	\]
	The last inequality implies for any \( u \geq \max(a, b) \),
	\[
	    \P(Y > u) \leq \exp\left(-\min\left\{ \frac{u^2}{a^2}, \frac{u}{b} \right\} \right) \, ,
	\]
	and vice versa.
	}
	
	\subsubsection*{Uniform Hanson-Wright inequalities under the logarithmic Sobolev condition}
	\label{logsobolev}
	
	In this paragraph we recover {a} result of \cite{Adam15} under  Assumption~\ref{logSobolev}. Consider a random {variable \( Z_{\AA}(X) \)} defined by  \eqref{zdef}, a function of \( X \) that satisfies logarithmic Sobolev assumption \eqref{log-sobol_assume}.
	
	Following \cite{Adam15} we assume without loss of generality, that \( \mathcal{A} \) is a finite set of matrices. Then \( Z_{\AA} \) is Lebesgue-a.e. differentiable and
	\[
	\| \nabla Z_{\AA}(X) \| \leq 2 \sup_{A \in \AA} \| A X \|,
	\]
	bounded by a Lipschitz function of \( X \) with good concentration properties. 
	
	\begin{remark}
		Note that Assumption~\ref{logSobolev} applies only for smooth functions, so that a standard smoothing argument should be used (see e.g.  \cite{Ledoux2001}). For the sake of completeness we recall this argument in Section~\ref{smooth_approx:section}. In what follows in this section we assume that none of these potential technical problems appear.
	\end{remark}
	
	In particular, since \( X \) satisfies {the} logarithmic Sobolev condition with constant \( K \), we have by Theorem 5.3 in \cite{Ledoux2001} that
	\[
	\P\left( \sup_{A \in \AA} \| A X \| \ge \E\sup_{A \in \AA} \| A X \| + K \sqrt{t} \sup_{A \in \AA} \| A \|  \right) \leq e^{-t}.
	\]
	Taking square{s} and using \( (a + b)^{2} \leq 2a^{2} + 2b^{2} \) we get
	\[
	\P\left( \sup_{A \in \AA} \| A X \|^{2} \ge 2\left(\E\sup_{A\in\AA} \| A X \|\right)^{2} + 2 K^2 \sup_{A\in\AA} \| A \|^2 t  \right) \leq e^{-t}.
	\]
	Furthermore, the logarithmic Sobolev condition implies for any \( \lambda \in \R \)
	\[
	\Ent(e^{\lambda Z_{\AA}(X)}) \leq 4K^{2}\lambda^2 \E \sup_{A\in\AA} \| AX \|^{2} e^{\lambda Z_{\AA}(X)}.
	\]
	Therefore, by Lemma~\ref{grad_trunc_lem} it holds for any \( t \geq 0 \) that
	\[
	\P\left( | Z_{\AA}(X) - \E Z_{\AA}(X)| > C \left( K \E\sup_{A\in\AA} \| A X \| \sqrt{t} + K^{2} \sup_{A\in\AA} \| A \|t  \right)  \right) \leq 2e^{-t},
	\]
	which coincides with \eqref{adam} for \( K \)-concentrated vectors up to absolute constant factors.
	\begin{remark}
		This result may be used directly to prove the concentration for $\|\hat{\Sigma} - \Sigma\|$, where $\hat{\Sigma}$ is the sample covariance defined as $\hat{\Sigma} = \frac{1}{N}\sum_{i = 1}^{N}X_iX_i^\T$ and $X_1, \ldots, X_N$ are centered Gaussian vectors with the covariance matrix $\Sigma$ (see Theorem 4.1 in \cite{Adam15}). We return to the covariance estimation problem in {Section \ref{matrixbern}}.
	\end{remark}
	
	\begin{remark}
		{We note some additional connections between the convex concentration property \eqref{concproperty} and Assumption \ref{logSobolev}. It is known that \eqref{concproperty} follows from the logarithmic Sobolev inequality by taking $f(X)= \exp(\lambda (\varphi(X) - \E\varphi(X))/2)$ for $\lambda > 0$ which implies
	\[
	\Ent\left(\exp(\lambda (\varphi(X) - \E\varphi(X)))\right) \le \frac{K^2\lambda^2}{2}\E\exp(\lambda (\varphi(X) - \E\varphi(X))).
	\]
	This immediately implies \eqref{concproperty} via {the} standard Herbst argument, see \cite{boucheron2013concentration}. Moreover, the last inequality is equivalent to {the} concentration property. Indeed, from the concentration property we know that $\|\varphi(X) - \E\varphi(X)\|_{\psi_2} \lesssim K$ and this implies (see \cite{Vanhandel16}) that for all $\lambda \in \mathbb{R}$
	\[
	\Ent(\exp(\lambda (\varphi(X) - \E\varphi(X)))) \lesssim K^2\lambda^2\E\exp(\lambda (\varphi(X) - \E\varphi(X))).
	\]
	}
	\end{remark}
	
	\subsubsection*{Improving Hanson-Wright inequality in the subgaussian regime}
	\label{boundedcase}
	
	Our analysis implies, in particular, an improved version of Hanson-Wright inequality \eqref{HansonWright} in some cases. We consider a centered random vector $X = (X_1, \ldots, X_n)$ with independent subgaussian components and set $K = \max_{i}\|X_i\|_{\psi_2}$, $M = \|\max_{i}|X_i|\|_{\psi_2}$. In this case \eqref{HansonWright} implies that with probability at least $ 1 - 2e^{-t} $ we have
	\begin{equation}
	\label{firstineq}
	X^\T AX - \E X^\T AX \lesssim K^2\left(\|A\|_{\text{HS}}\sqrt{t} +  \| A \| t\right).
	\end{equation}
	At the same time, Theorem \ref{mainthm} for a single matrix \( \mathcal{A} = \{ A \} \) implies with the same probability
	\begin{equation}
	\label{secineq}
	X^\T AX - \E X^\T AX \lesssim  M\E\| A X \|  \sqrt{t} + M^2\| A \| t.
	\end{equation}
	Observe that when $|X_i| \le L$ almost surely for every \( i \leq n \), we have $M \lesssim \min\{K\sqrt{\log n}, L\}$.
	The following example illustrates the difference between these two bounds.
	\begin{example}
		Assume, {$\delta_{1}, \ldots, \delta_n$ are} independent Bernoulli random variables with the same mean $\delta$ and let $\delta \le \frac{1}{4}$. For $X = (\delta_1 - \delta, \ldots, \delta_n - \delta)$ we easily get 
		\[
		\E\| A X \| \le \sqrt{\E X^TA^2X} \leq  \sqrt{\delta}\|A\|_{\text{HS}}.
		\]
		On the other hand, for $\delta \le \frac{1}{4}$ we have
		\begin{align*}
		\|X_1\|_{\psi_2}^2 &= \|\delta_1 - \delta\|^2_{\psi_2} \sim \sup\limits_{\lambda \in \mathbb{R}}\frac{\log(\E\exp(\lambda(\delta_1 - \delta)))}{\lambda^2} 
		\\
		&= \sup\limits_{\lambda \in \mathbb{R}}\frac{\log(\delta\exp(\lambda(1 - \delta)) + (1 - \delta)\exp(-\lambda\delta))}{\lambda^2}
		=\frac{1 - 2\delta}{4\log((1 - \delta)/\delta)} \sim \frac{1}{|\log \delta|},
		\end{align*}
		where the last line follows directly from Theorem 1.1 in \cite{schlemm16} {(a result equivalent to Theorem 1.1 was also obtained in \cite{kontorovich13})}. Therefore, the standard Hanson-Wright inequality implies that with probability at least $1 - e^{-t}$ we have
		\[
		X^\T AX - \E X^\T AX \lesssim \frac{1}{|\log \delta|}\left(\|A\|_{\text{HS}}\sqrt{t} +   \| A \| t\right),
		\]
		while \eqref{secineq} and $M \lesssim \min\{K \sqrt{\log n}, 1\}$ imply that for $t \ge 1$ and $\delta \le \frac{1}{4}$ it holds with probability at least $1 - 2e^{-t}$ that
		\begin{equation}
		\label{improved}
		X^\T AX - \E X^\T AX \lesssim \min\left\{\sqrt{\frac{\delta\log n}{|\log \delta|}}, \sqrt{\delta}\right\}\|A\|_{\text{HS}}\sqrt{t} + \min\left\{\frac{\log n}{|\log \delta|}, 1\right\}\| A \|t.
		\end{equation}
		It is easy to verify that $\lim\limits_{\delta \to 0+}\sqrt{\delta}|\log \delta| = 0$, thus the inequality \eqref{improved} is better than Hanson-Wright inequality for this $X$ in the subgaussian regime (when the $t$-term is dominated by the $\sqrt{t}$-term).
	\end{example}
	
	\subsubsection*{Uniform concentration results in the Ising model}
	\label{ising}
	
	\def\hv{\mathbf{h}}
	\def\ddiff{\mathfrak{d}}
	\def\cond{\; | \;}
	
	Consider a random vector \( \sigma \in \{ -1, 1 \}^{n} \) with the distribution defined by
	\[
	\pi( \sigma ) = \frac{1}{Z^{\prime}} \exp\left( \sum_{i,j = 1}^{n} J_{ij} \sigma_{i} \sigma_{j} - \sum_{i = 1}^{n} h_{i} \sigma_{i}   \right),
	\]
	where \( Z^{\prime} \) is a normalizing factor. This distribution defines the \emph{Ising model} with parameters \( J = (J_{ij})_{i,j = 1}^{n} \)  and \( \hv = (h_i)_{i = 1}^{n} \). For an arbitrary function \( f \) on \( \{ -1, 1 \}^{n} \) denote a difference operator,
	\[
	| \ddiff f|^{2}(\sigma)  = \frac{1}{2} \sum_{i = 1}^{n} (f(\sigma) - f(T_i \sigma))^{2} \pi(-\sigma_i \cond \sigma_1, \dots, \sigma_{i-1}, \sigma_{i + 1}, \dots   ),
	\]
	where the operator \( T_{i} \sigma = (\sigma_1, \dots, \sigma_{i-1}, -\sigma_{i}, \sigma_{i + 1}, \dots  ) \) flips the sign of the \(i\)-th component, and \( \pi( \cdot \cond \sigma_1, \dots, \sigma_{i-1}, \sigma_{i + 1}, \dots   )  \) is conditional distribution of the \( i \)-th component given the rest of the elements. The following recent result provides the logarithmic Sobolev inequality for \( \sigma \) under Dobrushin-type conditions.
	
	\begin{theorem}[Proposition 1.1, \cite{Gotze18}]
		\label{gotze_logsob:thm}
		Suppose, \( \| \hv \|_{\infty} \leq \alpha \) and \( J \) satisfies \( J_{ii} = 0 \) and
		\begin{equation}
		\label{dobrushin}
		\| J \|_{1 \mapsto 1} = \max_{i =1,\dots, n} \sum_{j=1}^{n} |J_{ij}| \leq 1 - \rho 
		\end{equation}
		There is a constant \( C = C(\alpha, \rho) \), such that for an arbitrary function \( f \) on \( \{ -1,1\}^{n} \) we have
		\[
		\Ent( f^2 ) \leq 2C \E |\ddiff f|^2 .
		\]
	\end{theorem}
	\begin{remark}
		Following \cite{Gotze18} the condition \eqref{dobrushin} will be called \emph{Dobrushin's condition}.
	\end{remark}
	We may obtain the following uniform concentration result which is a simple outcome of our Lemma~\ref{grad_trunc_lem} and Theorem \ref{gotze_logsob:thm}.
	\begin{proposition}
		Let \( \mathcal{A} \) be a finite set of symmetric matrices with zero diagonal. It holds in the Ising model under Dobrushin's condition and \( \| \hv \|_{\infty} \leq \alpha \) that for any $t \ge 0$
		\begin{equation}
		\label{finalresult}
		\P\left( \sup_{A \in \mathcal{A}} \sigma^{\T} A \sigma  - \E \sup_{A \in \mathcal{A}} \sigma^{\T} A \sigma  \ge t \right) \le \exp\left(-C\min\left(\frac{t^2}{(\E\sup\limits_{A \in \mathcal{A}}\|A\sigma\| + \sup_{A \in \mathcal{A}} \| A \|)^2}, \frac{t}{\sup\limits_{A \in \mathcal{A}}\|A\|}\right)\right),
		\end{equation}
		where \( C \) depends only on \( \alpha, \rho \).
	\end{proposition}
	\begin{proof}
		Let \( \sigma_{(i)}' = ( \sigma_1, \dots, \sigma_{i-1}, \sigma_i', \sigma_{i + 1}, \dots ) \), where given all but the \(i\)-th element of \( \sigma \), the variables \( \sigma_{i} \) and \( \sigma_{i}' \) are independent and are distributed according to \( \pi( \cdot \cond \sigma_1, \dots, \sigma_{i-1}, \sigma_{i + 1}, \dots   ) \). Obviously, we may have all \( \sigma_{1}, \dots, \sigma_{i} \) and \( \sigma_{1}', \dots, \sigma_{n}' \) defined on the same discrete probability space, and thus we will use the notation \( \pi(\cdot) \) and \( \pi(\cdot \cond \cdot) \) for the distribution and the conditional distribution. Therefore, we have
		\begin{align*}
		\E | \ddiff f|^{2}(\sigma)   & =  \frac{1}{2} \sum_{i = 1}^{n} \E (f(\sigma) - f(T_i \sigma))^{2} \pi(-\sigma_i \cond \sigma_1, \dots, \sigma_{i-1}, \sigma_{i + 1}, \dots   ) \\
		& =  \sum_{i = 1}^{n} \sum_{\sigma \in \{-1, 1\}^n} \pi(\sigma) \sum_{\sigma_{i}' \in \{-1, 1\}} (f(\sigma) - f(\sigma_{(i)}'))_{+}^{2} \pi( \sigma_i' \cond \sigma_1, \dots, \sigma_{i-1}, \sigma_{i + 1}, \dots   )  
		\end{align*}
		where we switched from \( \frac{1}{2} (f(\sigma) - f(\sigma_{(i)}'))^{2} \) to \( (f(\sigma) - f(\sigma_{(i)}'))_{+}^{2} \) due to the symmetry between \( \sigma_i \) and \( \sigma_i' \). 
		
		Denoting \( \sigma^{-i} = (\sigma_1, \dots, \sigma_{i-1}, \sigma_{i + 1}, \dots,   \sigma_n) \) and using the independence of \( \sigma_{i} \) and \( \sigma_{i}' \) given \( \sigma^{-i} \) we observe that \( \pi(\sigma_{i}, \sigma_{i}' \cond \sigma^{-i}) = \pi(\sigma_{i} \cond \sigma^{-i}) \pi( \sigma_{i}' \cond \sigma^{-i} )  \). {Moreover, it follows from the definition of conditional probability that}
		\begin{align*}
		\pi(\sigma) \pi( \sigma_i' \cond \sigma_1, \dots, \sigma_{i-1}, \sigma_{i + 1}, \dots   ) &= \pi( \sigma^{-i} ) \pi(\sigma_{i} \cond \sigma^{-i}) \pi( \sigma_{i}' \cond \sigma^{-i} ) \\
		&= \pi( \sigma^{-i} ) \pi(\sigma_{i}, \sigma_{i}' \cond \sigma^{-i}) = \pi(\sigma_{i}', \sigma_{i}, \sigma^{-i} ).
		\end{align*}
		Finally, we get
		\[
		\E | \ddiff f|^{2}(\sigma) = \sum_{i = 1}^{n} \sum_{ (\sigma, \sigma_{i}') \in \{-1, 1\}^{n+1} } (f(\sigma) - f(\sigma_{(i)}'))^{2}_{+} \pi(\sigma, \sigma_i') = \sum_{i = 1}^{n} \E (f(\sigma) - f(\sigma_{(i)}'))^{2}_{+} \, .
		\]
		Now we want to consider the function
		\begin{equation}
		\label{Z_ising}
		{Z_{\AA}(\sigma)} = \sup_{A \in \mathcal{A}} \sigma^{\T} A \sigma,
		\end{equation}
		where \( \mathcal{A} \) is a given {finite} set of symmetric matrices with zero diagonal (the diagonal is not important here, since \( \sigma_i^2 = 1 \)).
		{Let us apply Theorem \ref{gotze_logsob:thm} to \( f(\sigma) = e^{\lambda Z_{\AA}(\sigma) / 2} \). Since for \( x \geq y \) and \( \lambda \geq 0\) we have \( (e^{\lambda x / 2} - e^{\lambda y/2})^2 = e^{\lambda x} (1 - e^{-\lambda (x - y)/2})^2 \leq \frac{\lambda^2}{4} e^{\lambda x} ( x  -  y)^2\), it holds that}
		\begin{align*}
		\E | \ddiff f|^{2}(\sigma) &= \E \sum_{i = 1}^{n} (f(\sigma) - f(\sigma_{(i)}'))_{+}^{2} = \E e^{\lambda Z_{\AA}(\sigma)} \sum_{i = 1}^{n} (1 - e^{-\lambda (Z_{\AA}(\sigma) - Z_{\AA}(\sigma_{(i)}'))/2}   )_{+}^{2} \\
		& \leq 
		\frac{\lambda^2}{4} \E e^{\lambda Z_{\AA}(\sigma)} \sum_{i = 1}^{n} (Z_{\AA}(\sigma)  - Z_{\AA}(\sigma_{(i)}'))_{+}^{2},
		\end{align*}
		where for \( \widetilde{A} \) (maximizer of \eqref{Z_ising}) we have,
		\begin{align*}
		\sum_{i = 1}^{n} (Z_{\AA}(\sigma)  - Z_{\AA}(\sigma_{(i)}'))_{+}^{2} & \leq  \sum_{i = 1}^{n} \left( \sigma^{\T} \widetilde{A} \sigma - [\sigma_{(i)}']^{\T} \widetilde{A} \sigma_{(i)}'  \right)_{+}^{2}
		=
		\sum_{i = 1}^{n} \left( 2 (\sigma_i - \sigma_i') \sum_{j = 1}^{n} \widetilde{A}_{ij} \sigma_{j}  \right)_{+}^{2} \\
		& \leq  16 \sup_{A \in \mathcal{A}} \| A \sigma \|^2.
		\end{align*}
		Note that concentration for \( \sup_{A \in \mathcal{A}} \| A \sigma \| \) is implied by the same result. Indeed, we have
		\begin{align*}
		\sum_{i = 1}^{n} \left(\sup_{A \in \mathcal{A}, \gamma \in S^{n - 1}} \gamma^{\T} A \sigma  - \sup_{A \in \mathcal{A}, \gamma \in S^{n - 1}} \gamma^{\T} A \sigma_{(i)}' \right)_{+}^{2}
		& \leq 
		\sum_{i = 1}^{n} (\tilde{w}^{\T}\sigma - \tilde{w}^{\T} \sigma_{(i)}')_{+}^{2} \\
		& =  \sum_{i = 1}^{n} (\tilde{w}_{i} (\sigma_i - \sigma_i'))_{+}^{2} \leq 4 \sup_{A\in\AA} \| A \|,
		\end{align*}
		where \( \widetilde{w}^{\T} = \gamma^{\T} A \) is such that \( \sup_{A\in \mathcal{A}} \| A \sigma \| = \widetilde{w}^{\T} \sigma \). Thus, the expectation of the corresponding difference operator is bounded by \( 4 \sup_{A\in \mathcal{A}} \| A \| \). Therefore, due to {the} standard Herbst argument {(Proposition~{6.1} in \cite{boucheron2013concentration})} Theorem~\ref{gotze_logsob:thm} implies
		\[
		\P\left( \sup_{A \in \mathcal{A}} \| A \sigma \| > \E \sup_{A \in \mathcal{A}} \| A \sigma \| + C \sup_{A\in \mathcal{A}} \| A \| \sqrt{t}  \right) \leq e^{-t}.
		\]
		To sum up, by Theorem~\ref{gotze_logsob:thm} we have
		\[
		\Ent( e^{\lambda Z_{\mathcal{A}}(\sigma)} ) \leq \lambda^{2} \E (4 \sup_{A \in \mathcal{A}} \| A \sigma \| ) e^{\lambda Z_{\mathcal{A}}(\sigma)}.
		\]
		It is left to apply Lemma~\ref{grad_trunc_lem} which finishes the proof of the following inequality
		\begin{equation}
		\label{finalresult}
		\P\left( \sup_{A\in\AA} \sigma^{\T} A \sigma  - \E \sup_{A\in\AA} \sigma^{\T} A \sigma  > C (\sqrt{t} \E \sup_{A\in\AA} \| A \sigma\| + (\sqrt{t} + t) \sup_{A\in\AA} \| A \| )    \right) \geq 1 - e^{-t},
		\end{equation}
		where \( C \) only depends on \( \alpha, \rho \) from Theorem~\ref{gotze_logsob:thm}. The claim follows.
	\end{proof}
	\begin{remark}
		In the case that $\mathcal{A} = \{A\}$ our result implies the upper tail of the recent concentration inequality proved in \cite{Adam18} (see Theorem 2.2 and Example 2.5). To show this fact (denoting $\overline{\sigma} = \sigma - \E\sigma$) we observe that 
		\[
		\E\| A \sigma\| \le \E \| A \overline{\sigma}\| + \|A\E\sigma\| = \E \| A \overline{\sigma}\|  + \bigl(\sum_{i = 1}^n(\sum_{j = 1}^nA_{i, j}\E\sigma_j)^2\bigr)^{\frac{1}{2}}.
		\]
		Now, it is well known that $\Ent( f^2 ) \leq 2c \E |\ddiff f|^2$ implies the Poincare {inequality} $\mathrm{Var}( f ) \leq c \E |\ddiff f|^2$. Therefore, we have
		\[
		\|\E \overline{\sigma}\ \overline{\sigma}^\T\| = \sup\limits_{u \in S^{n - 1}}\mathrm{Var}(u^T\overline{\sigma}) \le (c(\alpha, \rho)/2)\sup\limits_{u \in S^{n - 1}}4\|u\|^2 = 2c(\alpha, \rho).
		\]
		This implies,
		\[
		\E \| A \overline{\sigma}\|^2  = \E\tr(A^2\overline{\sigma}~\overline{\sigma}^\T) \le \|A\|^2_{HS}\|\E \overline{\sigma}~\overline{\sigma}^T\| \le 2c(\rho, \alpha) \|A\|^2_{HS},
		\]
		where we used that $\tr(BD) \le \tr(B)\|D\|$ which holds for any pair of symmetric and nonnegative matrices $B, D$. Finally, we have
		\[
		{\E}\| A \sigma\| \le C(\rho, \alpha)\|A\|_{\text{HS}} + \left(\sum_{i = 1}^n\left(\sum_{j = 1}^nA_{i, j}\E\sigma_j\right)^2\right)^{\frac{1}{2}}.
		\]
		The right-hand side term appears instead of ${\E}\| A \sigma\|$ in aforementioned Example 2.5.
	\end{remark}
	
	\subsubsection*{Replacing \( \| \max_{i} |X_i| \|_{\psi_{2}} \) with \( \max_{i} \| X_{i} \|_{\psi_2} \) in Theorem~\ref{mainthm} {and Lemma \ref{convexunbound}}}
	\label{logarithm}
	
	Here we show that it is essentially not possible to substitute \( \| \max_{i} |X_i| \|_{\psi_{2}} \) with \( \max_{i} \| X_{i} \|_{\psi_2} \) in Theorem~\ref{mainthm} by presenting a concrete counterexample which was kindly suggested by Rados{\l}aw Adamczak. 
	Suppose the opposite: there is an absolute constant \( C > 0 \) such that for any set of matrices \( \mathcal{A} \) and any subgaussian random variables \( X_1, \dots, X_n \) it holds with probability at least \( 1 - e^{-t} \) that
	\begin{equation}\label{counter_assume}
	Z_{\mathcal{A}}(X) \leq C \left(\E Z_{\mathcal{A}}(X) + \max_{i} \| X_{i} \|_{\psi_2} \sqrt{t} \E \sup_{A\in\AA} \| A X \| + \max_{i} \| X_{i} \|_{\psi_2}^2 \sup_{A\in\AA}\| A \| t   \right),
	\end{equation}
	which implies that for some other constant \( C' > 0 \) we have
	\[
	\E^{1/2} Z_{\mathcal{A}}(X)^{2} \leq C' \left(\E Z_{\mathcal{A}}(X) + \max_{i} \| X_{i} \|_{\psi_2} \E \sup_{A\in\AA} \| A X \| + \max_{i} \| X_{i} \|_{\psi_2}^2 \sup_{A\in\AA}\| A \| \right).
	\]
	Notice that here we allow a multiplicative constant not equal to $1$ in front of the expectation. Let us take \( \mathcal{A} = \{ A^{(1)}, \dots, A^{(n)} \} \) with \( A^{(i)} \) having only one nonzero element \( A^{(i)}_{ii} = 1 \). For the sake of simplicity we take i.i.d. \( X_1, \dots, X_n \) with \( \E X_{i}^{2} = 1 \). This implies
	\[
	Z_{\mathcal{A}}(X)  
	= \max_{i \leq n} (X_{i}^{2} - 1),
	\qquad
	\sup_{A \in \AA} \| A X \| = \max_{i \leq n} |X_i| ,
	\qquad
	\sup_{A \in \AA} \| A \| = 1.
	\]
	Assuming that \( \| X_1 \|_{\psi_2} \leq 4 \) we have
	\[
	\bigl\| \max_{i \leq n} X_{i}^{2} - 1 \bigr\|_{L_2} \leq C' \left(\E \max_{i \leq n} (X_{i}^{2} -1) + 4 \E \max_{i \leq n} |X_i| + 16 \right),
	\]
	which, since \( \| \max_{i \leq n} X_{i}^{2} \|_{L_1} \geq \max_{i \leq n} \| X_i \|_{L_2} = 1  \), implies
	\[
	\| \max_{i \leq n} X_{i}^{2} \|_{L_2} \leq 1 + C' (\| \max_{i \leq n} X_{i}^{2} \|_{L_1} + 4 \E \max_{i \leq n} |X_i| + 15 )
	\leq (1 + 20 C') \| \max_{i \leq n} X_i^2 \|_{L_1}.
	\]
	Note that this inequality also holds if we rescale \( X_i' = \alpha X_i \) for an arbitrary \( \alpha > 0 \). Therefore, if \( \| X_1 \|_{\psi_2} \leq 4 \| X_1 \|_{L_2}  \), we can always rescale our random variables to have \(  \| X_1 \|_{L_2}= 1 \) and \( \| X_1 \|_{\psi_2} \leq 4  \), so that the above inequality still holds.
	
	Taking the latter into account we conclude that there is a constant \( D > 0 \), such that if a centered random \( X_1 \) satisfies \( \| X_1 \|_{\psi_2} \leq 4 \| X_{1} \|_{L_2} \), then for any \( n \geq 1 \) the following inequality holds
	\begin{equation} \label{hypercontr}
	\| \max_{i \leq n} X_{i}^{2} \|_{L_2} \leq D \| \max_{i \leq n} X_{i}^2\|_{L_1} .
	\end{equation}
	
	It is known that such hypercontractivity of maxima implies certain regularity of tails of \( X_1^2 \). In this case by Theorem~4.6 in \cite{Hitczenko98} for any \( \rho, \varepsilon > 0 \) there is another constant \( A = A(D, \rho, \varepsilon) > 1 \) such that for every \( t \geq t_{0} = \rho \| X_1 \|_{L_1} \) we have
	\[
	A \P(X_1^2 > A t ) \leq  \varepsilon \P(X_1^2 > t),
	\]
	{so that} taking \( \rho = \varepsilon = 1 \), there is \( A = A(D) > 1 \) such that for all \( t \geq \| X_1 \|_{L_1} \) we have
	\begin{equation}\label{tail_regu}
	\P(X_1^2 > At) \leq \frac{1}{A} \P (X_1^2 > t).
	\end{equation}
	The latter does not have to hold for {every} subgaussian random variable \( X_1 \). For instance, taking a symmetric random variable \( X_1 \) with \( \P(|X_1| = 1) = 1 - e^{-r} \) and  \( \P(|X_1| = \sqrt{r}) = e^{-r} \)  for \( r \geq 4 > 4 \log 2 \) we have
	$
	\E \exp \left(\frac{|X_1|^2}{2} \right) = e^{\frac{1}{2}} (1 - e^{-r}) + e^{-r + \frac{r}{2}} \leq e^{\frac{1}{2}} + e^{-\frac{r}{2}} \leq 2,
	$
	which implies \( \| X_1 \|_{\psi_2} \le \sqrt{2} \). Moreover, for \( r \geq 4 \) we also have
	$
	\E X_1^2 \geq 1 - e^{-\frac{r}{2}} \geq \frac{1}{2},	
	$
	thus \( \| X_1 \|_{L_2} \geq 1/ \sqrt{2} \) and the conditions of \eqref{hypercontr} are satisfied. But for large enough \( r > At \) and for \( t = t_{0}\), we have
	\[
	\P\left( X_1^2 > At \right) = \P(X_1^2 > t) = e^{-r},
	\]
	therefore breaking the tail regularity \eqref{tail_regu}.
	Therefore, it is impossible to establish {an} inequality of the form \eqref{counter_assume}. We note that it is also possible to prove that \eqref{hypercontr} may not hold for $X_1$ defined above via some direct {calculations}.
	
	For the same reason it is not possible to replace \( \| \max_{i \leq n} |X_i| \|_{\psi_2} \) with \( \max_{i \leq n} \| X_i \|_{\psi_2} \) in Lemma~\ref{convex_subgauss:lemmamain}. Indeed, suppose that for any convex \( L \)-Lipschitz function \( f \) we have
	\[
	\P\left( f(X) \leq C (\E f(X) + L \max_{i \leq n} \| X_i \|_{\psi_2} \sqrt{t}  ) \right) \leq e^{-t} .
	\]
	Taking \( f(X) = \max_{i \leq n} |X_i| \), which is convex and \(1\)-Lipschitz, we get
	\[
	\bigl\| \max_{i \leq n} X_i^2 \bigr\|_{L_2} = \bigl\| \max_{i \leq n} |X_i| \bigr\|_{L_4} \leq C' \left( \E \max_{i\leq n} |X_i| + \max_{i\leq n} \| X_i \|_{\psi_2}  \right).
	\]
	The same choice of $X_1$ implies \eqref{hypercontr} and leads to a contradiction.
	
	\section{Proof of Theorem \ref{mainthm}}
	\label{expef}
	
	\def\at{\tilde{a}}
	\def\gammavt{\tilde{\gammav}}
	\def\At{\widetilde{A}}
	
	In this section we assume that the components of $X$ are independent. We recall that $X_i'$ denotes an independent copy of the component $X_i$. The main tool of the proof is the modified logarithmic Sobolev inequality (see Theorem 2 in \cite{Boucheron03} or Theorem~{6.15} in \cite{boucheron2013concentration}). {For the sake of brevity we denote \( Z = Z_{\AA}(X) \) in this section. Let us set}
	\[
	Z_i' = Z_{\AA}(X^{(i)}), \qquad X^{(i)} = (X_1, \dots, X_{i-1}, X_{i}', X_{i}, \dots, X_n).
	\]
	Then by the symmetrized version of the inequality we have that for any \( \lambda \),
	\[
	\Ent(e^{\lambda Z}) \leq  \sum_{i = 1}^{n} \E e^{\lambda Z} \tau( -\lambda (Z - Z_{i}'  )_{+} ),
	\]
	where \( \tau(x) = x(e^{x}-1) \). Since \( \tau(x) \leq x^{2} \) for \( x \leq 0 \), we have for all \( \lambda \geq 0 \),
	\begin{equation}
	\label{mod_log-sob}
	\Ent(e^{\lambda Z}) \leq \lambda^{2} \E V_{+} e^{\lambda Z},
	\qquad V_{+} := \E' \sum_{i = 1}^{n} (Z - Z_i')_{+}^{2}.
	\end{equation}
	The right-hand side of the inequality can be ``decoupled'' by the variational entropy formula \eqref{varformula}, as it is done in the proof of Lemma \ref{grad_trunc_lem} {which} was presented in the introduction.
	
	\begin{proof}[Proof of Lemma \ref{grad_trunc_lem}]
		We have
		\[
		\Ent(e^{\lambda Z}) \leq \lambda^{2} L {\E} e^{\lambda Z} + \lambda^{2} \E (W - L)_{+} e^{\lambda Z}.
		\]
		Due to the deviation bound for \( W \) it holds for some absolute constant {\( C > 0 \) that}
		\[
		\E \exp\left( \frac{(W - L)_{+}}{C \theta}  \right) \leq e.
		\]
		Therefore, by \eqref{varformula} we have
		\[
		\E (W - L)_{+}/ (C \theta) e^{\lambda Z} \leq \E e^{\lambda Z} + \Ent(e^{\lambda Z}),
		\]
		which implies
		\[
		(1 - C\theta \lambda^{2}) \Ent(e^{\lambda Z}) \leq \lambda^{2} (L + C \theta) \E e^{\lambda Z}.
		\]
		By the standard Herbst argument (see e.g., Proposition~{6.1} in \cite{boucheron2013concentration}) we have for any \( 0 < \lambda \leq ({2}C \theta)^{-1/2} \),
		\[
		\log \E \exp( \lambda (Z - \E Z)  ) \leq 2(L + C \theta) \lambda^{2}.
		\]
		{This moment generating function bound is known to immediately imply the right-tail concentration bound (see the properties of subgamma random variables in \cite{boucheron2013concentration}).} Finally, if \eqref{mod_log-sob_1} holds for all \( \lambda \in \R \), the two sided inequality can be derived in the same way.
	\end{proof}
	
	\begin{remark}
		Note, there is as well a moment version of the modified logarithmic Sobolev inequality, see e.g., Theorem 2 in \cite{Boucheron05}. By this theorem it holds for all $q \ge 2$ that
		\[
		\| (Z - \E Z)_{+} \|_{L_q} \leq \sqrt{2 \kappa q} \| \sqrt{V_{+}} \|_{L_q},
		\]
		where \( \kappa < 2  \) is an absolute constant. Then if we have 
		\begin{equation}
		\label{subexpvplus}
		\| \sqrt{V_{+}} \|_{L_q} \leq \sqrt{L} + \sqrt{\theta q},
		\qquad
		\forall q \geq 2,
		\end{equation}
		which is equivalent to the second inequality in \eqref{mod_log-sob_1} up to absolute constant factors, then it holds for any \( q \geq 2 \)
		\[
		\| (Z - \E Z)_{+} \|_{L_q} \leq \sqrt{4Lq} + \sqrt{4 \theta} q.
		\]
		The last inequality implies \eqref{decoupled_bound} up to absolute constant factors. We note that similar moment computations were used in \cite{Boucheron05} to analyze the Rademacher chaos. Similarly, one can introduce the moment analog of the logarithmic Sobolev inequality (see equation 3 in \cite{Adam15a}):
		\[
		\|Z(X) - \E Z(X)\|_{L_q} \le K\sqrt{q}\||\nabla Z(X)|\|_{L_q},
		\]
		where $K > 0$ is a constant, $|\cdot |$ stands for the standard Euclidean norm and $q \ge 2$.
		Now, if it holds (which in some cases may be derived by the second application of the moment analog of the logarithmic Sobolev inequality)
		\[
		\||\nabla Z(X)|\|_{L_q} \le \E|\nabla Z(X)| + \||\nabla Z(X)| - \E|\nabla Z(X)|\|_{L_q} \leq \sqrt{L} + K\sqrt{\theta q},
		\qquad
		\forall q \geq 2
		\]
		then
		\[
		\|Z - \E Z\|_{L_q} \leq K(\sqrt{Lq} + K\sqrt{\theta}q),
		\]
		which implies the result similar to \eqref{decoupled_bound}.
	\end{remark}
	
	Finally, we establish a version of our result that requires neither that \(X_i \) is centered nor that $X_i$ has variance one. It can happen that \( \E X^{\T} A X \neq \tr(A) \), but in fact, the value we subtract does not really affect the concentration properties. In general we can consider the random variable
	\begin{equation}
	\label{z_arbitrary_var}
	Z = \sup_{A \in \AA} ( X^{\T} A X - g(A) ),
	\end{equation}
	where \( g : \R^{n \times n} \rightarrow \R \) is an arbitrary function.
	
	\begin{lemma}
		\label{bounded_with_diag:thm}
		Suppose that the components \(X_i\) are independent but not necessar{il}y centered, {and \( | X_i| \leq K \) almost surely.} Then for $Z$ defined by \eqref{z_arbitrary_var} and for any \( t\geq 1 \) it holds with probability at least $1 - e^{-t}$ that
		\[
		Z - \E Z \leq C \left(K ( \E\sup_{A\in\AA} \| A X \| + \E \sup_{A\in\AA} \| \Diag(A) X\|  ) \sqrt{t} + K^{2} \sup_{A\in\AA} \| A \| t \right)
		,
		\]
		where \( C \) is an absolute constant.
	\end{lemma}
	\begin{proof}
		Let \( \At \) be the matrix that maximizes $Z(X)$ given $X$. We have
		\begin{align*}
		\sum_{i \leq n} ( Z - Z_i')^{2}_{+} &\leq  \sum_{i \leq n} \left( X^{\T} \At X - [X^{(i)}]^{\T} \At X^{(i)} \right)^{2} \\
		& = \sum_{i \leq n} \left( 2 (X_i - X_{i}') \sum_{j \neq i} \at_{ij} X_j + \at_{ii} (X_i^{2} - X_{i}'^{2}) \right)^{2} \\
		& = 
		\sum_{i \leq n}  (X_i - X_{i}')^2\left(2 \sum_{j \neq i} \at_{ij} X_j + \at_{ii} (X_i + X_{i}') \right)^{2}\\
		& \leq 
		(2K)^{2} \sum_{i \leq n} \left( 2 \sum_{j} \at_{ij} X_j + {\at_{ii}} (X_i' - X_{i}) \right)^{2},
		\end{align*}
		where the last line follows from $|X_i - X_{i}'| \le 2K$. {The factor $2$ appears in the second line because \( \At\) is symmetric and thus \( X_{i}' \) is counted twice.} Applying the triangle inequality we get
		\[
		V_+ = \E' \sum_{i \leq n} ( Z - Z_i')^{2}_{+} \leq (2K)^{2} \E' \sup_{A \in \AA} (2 \| A X \| + \| \Diag(A) X \| + \| \Diag(A) X'\|)^{2},
		\]
		{where \( \E'[ \cdot ] = \E[ \cdot \vert\; X]\) denotes the expectation with respect to the variables \( X_1', \dots, X_n'\) only.}
		Thus,
		\[
		V_{+} \leq 12 K^{2} \left( {4}\sup_{A \in \AA} \| A X\|^{2} + \sup_{A \in \AA} \| \Diag(A) X \|^{2} + \E \sup_{A \in \AA} \| \Diag(A) X \|^{2} \right),
		\]
		{where we used $(a + b + c)^2 \le 3(a^2 + b^2 + c^2)$}.
		Since \( |X_i| \leq K \), we have by convex concentration for Lipschitz functions (see e.g. Theorem~{6.10} in \cite{boucheron2013concentration})
		\begin{equation}
		\label{sup_norm_A_X:eq}
		\P\left(  \sup_{A\in\AA} \| A X \| > \E \sup_{A\in\AA} \| A X \|  + 2 \sqrt{2} K \sup_{A\in \AA} \| A \| \sqrt{t}  \right) \leq e^{-t}.
		\end{equation}
		{Using \( (a + b)^{2} \leq 2a^2 + 2b^2 \) we have
		\begin{equation}
		\P\left(  \sup_{A \in \AA} \| A X \|^2 > 2 \left(\E  \sup_{A\in\AA} \| A X \|\right)^2  + 16 K^2 \sup_{A\in\AA} \| A \|^2 t  \right) \leq e^{-t}.
		\end{equation}
		Similar inequality holds for the term \( \sup_{A \in \AA} \| \Diag(A) X \|^2 \). 
		Moreover, by the Poincare inequality (Theorem~3.17 in \cite{boucheron2013concentration}) we have
		\begin{align*}
		    \E \sup_{A \in \AA} \| \Diag(A) X \|^{2} &= \left( \E \sup_{A \in \AA} \| \Diag(A) X \| \right)^{2} + \mathrm{Var}\left( \E \sup_{A \in \AA} \| \Diag(A) X \| \right) \\
		    & \leq \left( \E \sup_{A \in \AA} \| \Diag(A) X \| \right)^{2} + (2K)^2 \sup_{A \in \AA} \|\Diag(A) \|^{2} .
		\end{align*}
		Since \( \| \Diag(A) \| \leq \| A \| \), we get that for \( L \sim K^2 \left(\E \sup_{A \in \AA} \| A X \| + \E \sup_{A \in \AA} \| \Diag(A) X \|\right)^2 \) and \( \theta \sim K^4 \left(\sup_{A \in \AA} \| A \| \right)^2 \) we have}
		\[
		\P\left( V_{+} > L + \theta t \right) \leq e^{-t}.
		\]
		Therefore, due to the modified logarithmic Sobolev inequality \eqref{mod_log-sob} we can use Lemma~\ref{grad_trunc_lem}. This provides us with the inequality
		\[
		\P(Z - \E Z > C(\sqrt{L + \theta} \sqrt{t} + \sqrt{\theta} t)) \leq e^{-t},
		\]
		where we can neglect the \( \theta \) in front of \( \sqrt{t} \) when \( t \geq 1 \).
	\end{proof}
	
	Note that our bound has the term \( \E \sup_{A \in \AA} \| \Diag(A) X \| \) which can be avoided in the case of centered variables \( X_i \). Therefore, we obtain the bound matching the previous results \eqref{adam} and \eqref{Talagrand}.
	
	\begin{corollary}
		\label{bounded:corollary}
		Suppose that \( | X_i| \leq K \) almost surely and \( \E X_i = 0 \). Then for any \( t \geq 1 \) it holds with probability at least $1 - e^{-t}$ that
		\[
		    Z - \E Z {\leq C} \left(K \E\sup_{A\in\AA} \| A X \|  \sqrt{t} + K^{2} \sup_{A\in\AA} \| A \| t \right)
		,
		\]
		where \( C > 0 \) is an absolute constant. 
	\end{corollary}
	
	In the next two lemmas we show how to get rid of the diagonal term. This finishes the proof of the corollary above.
	
	\begin{lemma}
		\label{diag_quad_comp}
		Suppose that \( Y \in \R^{n} \) has the i.i.d. components with symmetric distribution and let \( \BB \) be a {finite} set of \( n \times n \) positive-semidefinite symmetric matrices. Then we have
		\[
		\E \sup_{B \in \BB} Y^{\T} \Diag(B) Y \leq  \E \sup_{B \in \BB} Y^{\T} B Y.
		\]
	\end{lemma}
	\begin{proof}
		{Since any \(B \in \BB \) is positive-semidefinite, \( \sup_{B \in \BB} x^{\T} B x \) is a convex function of \( x \in \R^{n} \).}
		Moreover, \( Y \overset{d}= \diag(\epsv) Y \) for an independent Rademacher vector \( \epsv \in \{ 1, -1 \}^n \). Therefore, by Jensen's inequality
		\begin{align*}
		\E \sup_{B \in \BB} Y^{\T} B Y &= \E \E_{\epsv} \sup_{B \in \BB} Y^{\T} \diag(\epsv)B \diag(\epsv) Y \\
		&\geq  \E \sup_{B \in \BB} \E_{\epsv} Y^{\T}  \diag(\epsv)B \diag(\epsv) Y \\
		&= \E \sup_{B \in \BB} Y^{\T} \Diag(B) Y,
		\end{align*}
		where \( \E_{\epsv} \) denotes the expectation with respect to $\epsv$ given $Y$.
	\end{proof}
	
	\begin{lemma}\label{diag_compare:lem}
		For \( X \) with the components that are independent and mean zero we have
		\[
		\E \sup_{A \in \AA} \| \Diag(A) X \| \leq C \E \sup_{A \in \AA} \| A X \|,
		\]
		where \( C >0 \) is an absolute constant.
	\end{lemma}
	\begin{proof}
		Let \( X' \) be an independent copy of \( X \). By the standard symmetrization argument together with Jensen' inequality and the triangle inequality we have
		\begin{equation}
		\label{symm_arg}
		\E \sup_{A \in \AA} \| A X \| \leq  \E \sup_{A \in \AA} \| A (X - X') \| \leq 2 \E \sup_{A \in \AA} \| A X \|.
		\end{equation}
		Observe that \( X - X' \overset{d}= (X- X') \diag(\epsv) = \diag(X - X')\epsv\) where \( \epsv \in \{ 1, -1 \}^n \) is an independent Rademacher vector. Therefore, we have
		\[
		\E \sup_{A \in \AA} \| A(X - X') \| = \E \E_{\epsv} \sup_{A \in \AA} \| A \, \diag(X - X') \epsv \|,
		\]
		where \( \E_{\epsv} \) denotes the expectation with respect to $\epsv$. 
		Conditionally on \( (X - X') \) set \( \AA_{X, X'} = \{ A \, \diag(X- X') \, : \, A \in \AA \} \). Let \( \av_1, \dots, \av_n \) be the columns of \( A \). Notice that for any matrix \(A\) we have \( \Diag(A^{\T}A) = \diag( \| \av_1 \|^2, \dots, \| \av_n\|^{2} ) \succeq \diag(A_{11}^{2}, \dots, A_{nn}^{2}) = \Diag(A)^{2} \). Therefore, by Lemma~\ref{diag_quad_comp} we have
		\begin{equation}
		\label{square_diag_all:eq}
		\E_{\epsv} \sup_{A \in \AA_{X, X'}} \| \Diag(A) \epsv \|^{2} \leq \E_{\epsv} \sup_{A \in \AA_{X, X'}} \| A \epsv \|^{2}.
		\end{equation}
		
		We now want to get rid of the squares in \eqref{square_diag_all:eq}. Let \( \mathcal{B} \) be an arbitrary set of symmetric \(n \times n \) matrices and let us fix some \( B \in \mathcal{B} \). We have \( \E \| B \epsv \|^{2} = \| B \|_{HS}^2 \) and by Khinchin's inequality we have
		\[
		\E \| B \epsv \| \geq \frac{1}{\sqrt{2}} \| B \|_{HS},
		\]
		with the optimal constant due to \cite{szarek1976best}. Thus, we have
		\[
		\E \sup_{B \in \mathcal{B}} \| B \epsv \| \geq \sup_{B \in \mathcal{B}} \E \| B \epsv \| \geq \frac{1}{\sqrt{2}} \sup_{B \in \mathcal{B}} \| B \|.
		\]
		Furthermore, by the convex Poincare inequality (Theorem 3.17, \cite{boucheron2013concentration}) we have,
		\[
		\mathrm{Var}( \sup_{B \in \mathcal{B}} \| B \epsv \| ) = \E \sup_{B \in \mathcal{B}} \| B \epsv \|^2 - \left( \E \sup_{B \in \mathcal{B}} \| B \epsv \| \right)^{2} \leq 4 \sup_{B \in \mathcal{B}} \| B \|^{2}.
		\]
		Therefore,
		$
		\E \sup_{B \in \mathcal{B}} \| B \epsv\|^{2} \leq \left( \E \sup_{B \in \mathcal{B}} \| B \epsv \| \right)^{2} + 4 \sup_{B \in \mathcal{B}} \| B \|^{2} \leq 9 \left( \E \sup_{B \in \mathcal{B}} \| B \epsv \| \right)^{2}
		$
		and we get
		\[
		(\E \sup_{B \in \mathcal{B}} \| B \epsv \|)^{2} \leq \E \sup_{B \in \mathcal{B}} \| B \epsv \|^{2} \leq 9 (\E \sup_{B \in \mathcal{B}} \| B \epsv \|)^{2}.
		\]
		The last inequality combined with \eqref{square_diag_all:eq} implies
		\[
		\E_{\epsv} \sup_{A \in \AA_{X, X'}} \| \Diag(A) \epsv \| \leq \left(\E_{\epsv} \sup_{A \in \AA_{X, X'}} \| \Diag(A) \epsv \|^{2}\right)^{\frac{1}{2}}
		\leq 3 \E_{\epsv} \sup_{A \in \AA_{X, X'}} \| A \epsv \|.
		\]
		Now, taking the expectation with respect to \( X, X' \) and applying \eqref{symm_arg} again we finish the proof.
	\end{proof}
	
	\subsection{Truncation for unbounded variables} 
	In this section we finish the proof of Theorem~\ref{mainthm}. In order to apply the bounded version of our inequality, we want to truncate each variable \( X_i \), which can be done by the approach from \cite{Adam08} (see references therein for more details on various applications of this method), where it was used in the context of Talagrand's concentration inequality. Suppose that \( \| \max_i |X_i| \|_{\psi_2} < \infty \) and set
	\begin{equation}\label{Y_W_truncation}
	Y_{i} = X_{i} \Ind( |X_i| \leq M ),
	\qquad
	W_{i} = X_{i} - Y_{i},
	\end{equation}
	with \( M = 8 \E \max |X_i| \). We have,
	\begin{align}
	Z_{\AA}(X) & =  \sup_{A\in\AA} ( Y^{\T} A Y - \E X^\T A X + W^{\T} A X + W^{\T} A Y )\nonumber \\ 
	& \leq  \sup_{A\in\AA} ( Y^{\T} A Y - \E X^\T A X) + \sup_{A\in\AA} | W^{\T} A X | + \sup_{A\in\AA} | W^{\T} A Y | \nonumber \\ 
	& \leq 
	\sup_{A\in\AA} ( Y^{\T} A Y - \E X^\T A X) + \| W \|  \sup_{A\in\AA} \| A X \| + \| W \| \sup_{A\in\AA} \| A Y \|. \label{ZY_ZX_comp}
	\end{align}
	The variables \( Y_i \) are now bounded by the value \( M \). Therefore, the first term of the last line can be analyzed by Lemma~\ref{bounded_with_diag:thm}. 
	
	To bound the rest we need to control the deviations of \( \| W \| \). We have
	$\| W \|^{2} = W_1^{2} + \dots + W_n^{2}$ is a sum of independent random variables with bounded \( \psi_1 \)-norm. Thus, we can control it's expectation via the Hoffman-J\o{}rgensen inequality. Due to the choice of the truncation level we have by Markov's inequality
	\[
	\P\left( \max_{i\leq n} W_i^{2} > 0  \right) = %
	\P\left( \max_{i\leq n} |X_i| > M \right) \leq \frac{\E \max\limits_{i\leq n} |X_i|}{M} \leq \frac{1}{8}.
	\]
	Denoting $S_{k} = W_{1}^{2} + \dots + W_{k}^{2}$ we have \( \| W \|^{2} = S_{n} \). Then,
	\[
	P \left( \max_{k \leq n} | S_{k} | > 0 \right) \leq P \left( \max_{i \leq n} W_i^2 > 0 \right) \leq \frac{1}{8}.
	\]
	Therefore, by Proposition~{6.8} in \cite{ledoux2013probability} we have
	\[
	\E \| W \|^{2} = \E S_n \leq 8 \E \max_{i \leq n} W_{i}^{2} \lesssim \| \max_{i \le n} |X_i| \|_{\psi_{2}}^{2},
	\]
	where the latter holds since \( \|\max\limits_{i \leq n} W_{i}^{2} \|_{\psi_{1}} \leq \| \max\limits_{i\leq n} |X_i| \|_{\psi_{2}}^{2} \).
	Furthermore, by Theorem~{6.21} in \cite{ledoux2013probability} we have
	\begin{align*}
	\left\| \sum_{i= 1}^n W_{i}^{2} - \E W_{i}^{2} \right\|_{\psi_1} & \leq  K_{1} \left( \E \bigl| \| W \|^{2} - \E \| W\|^{2} \bigr| + \bigl\| \max_{i \leq n} |W_{i}^2 - \E W_i^{2}| \bigr\|_{\psi_{1}}  \right) \\
	& \leq 
	2 K_{1} \left( \E \| W \|^{2}  + \bigl\| \max_{i\leq n} W_{i}^2 \bigr\|_{\psi_{1}}  \right) \\
	& \lesssim  \| \max_{i\leq n} |X_{i}| \|_{\psi_{2}}^{2},
	\end{align*}
	where $K_1$ is an absolute constant. Given the bound on the expectation of \(\| W \|^{2}  \) we have
	\[
	\bigl\|  \| W \| \bigr\|_{\psi_2} \lesssim \| \max_{i\leq n} |X_i| \|_{\psi_2}.
	\]
	Finally, we obtain the deviation bound: for every \( t > 0  \) we have
	\begin{equation}
	\label{smart_bound_for_Z}
	\P\left( \| W \| \geq C \sqrt{t} \| \max_{i\leq n} |X_{i}| \|_{\psi_{2}}   \right) \leq 2 e^{-t}.
	\end{equation}
	
	Now we apply Lemma~\ref{bounded_with_diag:thm} to the bounded variables \( Y\). Notice that our theorem does not require the variables to be centered. This assumption is only used in Corollary~{\ref{bounded:corollary}}. Taking this into account, Lemma \ref{bounded_with_diag:thm} can be applied to the variables \( Y \) as follows. Set \( g(A) = \E X^\T A X \) and \( Z_{\AA}(Y) = \sup_{A\in\AA} ( Y^{\T} A Y - g(A) ) \). By Lemma~\ref{bounded_with_diag:thm} we have
	\begin{equation}
	\label{ZY_conc}
	Z_{\AA}(Y) - \E Z_{\AA}(Y) \lesssim M \sqrt{t} \left(\E \sup_{A\in\AA} \| A Y \| + \E \sup_{A\in\AA} \| \Diag(A) Y \| \right) + M^{2} t \sup_{A\in\AA} \| A \|
	\end{equation}
	with probability at least \( 1 - e^{-t} \). Finally, we have to replace the expectations \( \E Z_{\AA}(Y) \), \( \E \sup\limits_{A\in\AA} \| A Y \| \) and \( \E \sup\limits_{A\in\AA} \| \Diag(A) Y \| \) in \eqref{ZY_conc} by their counterparts, taken with respect to \( X \), as in the original formulation of the result. 
	
	First, we want to provide a concentration bound for the convex function \( \sup\limits_{A\in\AA} \| A X \| \) that accounts for unbounded variables. As a matter of fact, we prove the following Lemma which is even slightly stronger than Lemma \ref{convex_subgauss:lemmamain}.
	
	\begin{lemma}
		\label{convex_subgauss:lem}
		Let \( f: \mathbb{R}^n \to \mathbb{R} \) be separately convex\footnote{This means that  for every $i=1,...,n$ it is a convex function of $i$-th variable if the rest of the variables are fixed. Any convex function is separately convex.} \( L \)-Lipschitz with respect to the Euclidean norm in \( \R^{n} \) and $X = (X_1, \ldots, X_n)$ be a random vector with the independent components. Then it holds for \( t \geq 1 \) that
		\[
		\P\left( f(X) > \E f(X) + C \bigl\| \max_{i\leq n} |X_i| \bigr\|_{\psi_2} L \sqrt{t}  \right) \leq e^{-t},
		\]
		where \( C > 0 \) is an absolute constant.
		Additionally, if $f$ is convex and \( L \)-Lipschitz, then for any \( t > 0 \),
		\[
		\P\left(  |f(X) - \E f(X)| > C \bigl\| \max_{i \leq n} |X_i| \bigr\|_{\psi_2} L \sqrt{t} \right) \leq 4 e^{-t}.
		\]
	\end{lemma}
	\begin{proof}
		By convex concentration (Theorem 6.10 in \cite{boucheron2013concentration}) for bounded \( Y_i \) defined by \eqref{Y_W_truncation} we have that for any \(  t > 0 \),
		\[
		\P\left(f(Y) > \E f(Y) + C \| \max_{i\leq n} |X_i| \|_{\psi_{2}} L \sqrt{t} \right) \leq e^{-t}.
		\]
		Moreover, due to the Lipschitz assumption and \eqref{smart_bound_for_Z} we have
		\[
		|f(X) - f(Y)| \leq L \| W \| \lesssim L \| \max_{i\leq n} |X_i| \|_{\psi_2} \sqrt{1 + t},
		\]
		where the latter holds with probability at least \( 1 - e^{-t} \).
		Integrating these two bounds we also get
		\begin{equation}\label{EfX_EfY}
		| \E f(X) - \E f(Y) | \lesssim L \| \max_{i\leq n} |X_i| \|_{\psi_2},
		\end{equation}
		which together implies that with probability at least \( 1 - e^{-t} \) we have
		\begin{align*}
		f(X) - \E f(X) & \leq  f(Y) - \E f(Y) + |f(X) - f(Y)| + |\E f(X) - \E f(Y)| \\
		& \lesssim 
		L \| \max_{i\leq n} |X_i| \|_{\psi_2} \sqrt{t}.
		\end{align*}
		The proof of the lower tail bound follows from Theorem 7.12 in \cite{boucheron2013concentration} and the standard relation between the median and the expectation which holds in our case.
	\end{proof}
	From Lemma \ref{convex_subgauss:lem} due to the fact that \( \sup\limits_{A\in\AA} \| A X \| \) is \( \sup\limits_{A\in\AA} \| A \| \)-Lipschitz we have
	\begin{equation}\label{norm_AX_conc_subgauss}
	\P\left(  \sup_{A \in \AA} \| A X \| > \E \sup_{A\in\AA} \| A X \| + C \| \max_{i\leq n} |X_i| \|_{\psi_2} \sup_{A\in\AA} \| A \| \sqrt{t} \right) \leq 2 e^{-t}.
	\end{equation}
	Moreover, similar to \eqref{EfX_EfY} we have
	\begin{equation}\label{EAYEAX_comp}
	\left|\E \sup_{A\in\AA} \| A Y \| - \E \sup_{A\in\AA} \| A X \| \right| \lesssim \| \max_{i\leq n} |X_{i}| \|_{\psi_{2}} \sup_{A\in\AA} \| A \|.
	\end{equation}
	Next, we bound the difference between \( \E Z_{\AA}(X) \) and \( \E Z_{\AA}(Y) \).
	
	\begin{lemma}\label{EZX_EZY_comp}
		We have
		\[
		|\E Z_{\AA}(Y) - \E Z_{\AA}(X)| \lesssim \| \max_{i\leq n} |X_{i}| \|_{\psi_{2}} \E \sup_{A\in\AA} \|A X \| + \bigl\| \max_{i\leq n} |X_{i}| \bigr\|_{\psi_{2}}^{2} \sup_{A\in\AA} \| A \| .
		\]
	\end{lemma}
	\begin{proof}
		Similarly to \eqref{ZY_ZX_comp} we have
		\begin{align}
		|\E Z_{\AA}(Y) - \E Z_{\AA}(X)| & \leq  \E \| W \|  \sup_{A\in\AA} \| A X \| + \E \| W \| \sup_{A\in\AA} \| A Y \|  \nonumber \\
		& \leq 
		\E^{1/2} \| W \|^{2} ( \E^{1/2} \sup_{A\in\AA} \| A X \|^{2} + \E^{1/2}  \sup_{A\in\AA} \| A Y \|^{2}),\label{ZY_ZX_diff}
		\end{align}
		where by \eqref{smart_bound_for_Z} \(\E^{1/2} \| W \|^{2} \lesssim \| \max_{i\leq n} |X_{i}| \|_{\psi_{2}}\) and by \eqref{norm_AX_conc_subgauss}, 
		\[
		\E \sup_{A\in\AA} \| A X \|^{2}  \lesssim  \left(\E \sup_{A\in\AA} \| A X \|\right)^{2} + \| \max_{i\leq n} |X_{i}| \|_{\psi_{2}}^{2} \sup_{A\in\AA} \| A \|^{2}.
		\]
		Taking the square root we get
		\[
		\E^{1/2} \sup_{A\in\AA} \| A X \|^{2}  \lesssim  \E \sup_{A\in\AA} \| A X \| + \| \max_{i\leq n} |X_{i}| \|_{\psi_{2}} \sup_{A\in\AA} \| A \|.
		\]
		Similarly and using {\eqref{EAYEAX_comp}} we have,
		\begin{align*}
		\E^{1/2} \sup_{A\in\AA} \| A Y \|^{2} & \lesssim  \E \sup_{A\in\AA} \| A Y \| + \| \max_{i\leq n} |X_{i}| \|_{\psi_{2}} \sup_{A \in \AA} \| A \| \\
		& \lesssim 
		\E \sup_{A\in\AA} \| A X \| + \| \max_{i\leq n} |X_{i}| \|_{\psi_{2}} \sup_{A\in\AA} \| A \|.
		\end{align*}
		Plugging it in \eqref{ZY_ZX_diff} we get the required inequality.
	\end{proof}
	
	Therefore, in \eqref{ZY_conc} we can use Lemma \ref{EZX_EZY_comp} to get
	\begin{equation}
	\label{EZY_EZX_via_AX}
	\E Z_{\AA}(Y)  \leq   \E Z_{\AA}(X) + C \left(\| \max_{i\leq n} |X_{i}| \|_{\psi_{2}} \E \sup_{A \in \AA} \|A X \| + \| \max_{i\leq n} |X_{i}| \|_{\psi_{2}}^{2} \sup_{A \in \AA} \| A \|\right),
	\end{equation}
	and by Lemma~\ref{EAYEAX_comp} (neglecting the diagonal term for centered \( X \) due to Lemma~\ref{diag_compare:lem})
	\begin{equation}\label{EAY_EAX}
	\E \sup_{A\in \AA} \| A Y \| + \E \sup_{A\in\AA} \| \Diag(A) Y \|  \leq 
	C\left(\E \sup_{A\in\AA} \| A X \| + \| \max_{i\leq n} |X_{i}| \|_{\psi_{2}} \sup_{A\in\AA} \| A \| \right).
	\end{equation}
	Finally, with probability at least $1 - e^{-t}$ for \( t \geq 1 \) we have from \eqref{ZY_ZX_comp}, {\eqref{EAYEAX_comp}} and \eqref{norm_AX_conc_subgauss}
	\begin{align*}
	|Z_{\AA}(X) - Z_{\AA}(Y)| & \leq  \| W \| \sup_{A\in\AA} \| A Y \| + \| W \| \sup_{A\in\AA} \| A X \| \\
	& \lesssim 
	\| W \| \E \sup_{A\in\AA} \| A X \| + \| W \| \| \max_{i\leq n} |X_i| \|_{\psi_2} \sup_{A\in\AA} \| A \| \sqrt{t},
	\end{align*}
	which using \eqref{smart_bound_for_Z} turns into
	\[
	|Z_{\AA}(X) - Z_{\AA}(Y)| \lesssim \| \max_{i\leq n} |X_i| \|_{\psi_2} \E \sup_{A\in\AA} \| A X \| \sqrt{t} + \| \max_{i\leq n} |X_i| \|_{\psi_2}^{2} \sup_{A\in\AA} \| A \| t.
	\]
	Combining the last inequality together with \eqref{EZY_EZX_via_AX} and \eqref{EAY_EAX} we finish the proof of Theorem~\ref{mainthm}.
	
	\subsection{Proof of Proposition \ref{mainthm2}}
	
	The proof is essentially based on the application of the next standard deviation bound instead of the concentration bound of \eqref{norm_AX_conc_subgauss} in the proof of Theorem \ref{mainthm}. Since we did not find an exact reference, we derive this inequality here.
	
	\begin{lemma}\label{norm_AX_by_Gauss}
		Suppose that \( X_1, \dots, X_n \) are independent centered random variables and \( \AA \) is a finite set of symmetric matrices. Let \( \gv  \) be a standard normal vector in \( \R^n\). 
		Then it holds with probability at least $1 - Ce^{-t}$ that
		\[
		\sup_{A \in \mathcal{A}}\|AX\| \lesssim \max_{i\leq n}\bigl\| X_i \bigr\|_{\psi_2}\left(\E \sup_{A \in \mathcal{A}} \| A \gv \| + \sup_{A\in\AA}\| A \| \sqrt{t}\right),
		\]
		where $C > 0$ is an absolute constant.
	\end{lemma}
	\begin{proof}
		At first, we observe that $
		\sup_{A \in \mathcal{A}}\|AX\| = \sup\limits_{A \in \mathcal{A}, \gamma \in S^{n - 1}}\gamma^TAX$. Consider the metric $\rho$ defined by $\rho(a, b) = \|a - b\|\max\limits_{i\leq n}\|X_i\|_{\psi_2}$ for any $a, b \in \mathbb{R}^n$. By Theorem 2.2.26 in \cite{Talagrand2014} it holds for $t \ge 0$ and an absolute constant $C > 0$ that with probability at least $1-C\exp(-t)$,
		\[
		\sup\limits_{A \in \mathcal{A}, \gamma \in S^{n - 1}}\gamma^TAX  \lesssim \text{diam}(\mathcal{A}S^{n - 1}, \rho)\sqrt{t} + \gamma_{2}(\mathcal{A}S^{n - 1}, \rho),
		\]
		where $\text{diam}(\mathcal{A}S^{n - 1}) = \sup\limits_{x, y \in \mathcal{A}S^{n - 1}}\|x - y\|\max\limits_{i\leq n}\|X_i\|_{\psi_2} \le 2\sup\limits_{A \in \mathcal{A}}\|A\|\max\limits_{i\leq n}\bigl\| X_i \bigr\|_{\psi_2}$ and the functional $ \gamma_{2}$ is also defined in \cite{Talagrand2014}. For the sake of brevity, we will not introduce its definition here. Finally, applying Talagrand's majorizing measure theorem (Theorem 2.4.1 in \cite{Talagrand2014}) we have
		\[
		\gamma_{2}(\mathcal{A}S^{n - 1}, \rho) \lesssim \max_{i\leq n}\bigl\| X_i \bigr\|_{\psi_2}\E\sup\limits_{x \in \mathcal{A}S^{n - 1}}x^TG = \max_{i\leq n}\bigl\| X_i \bigr\|_{\psi_2}\E\sup\limits_{A \in \mathcal{A}}\|AG\|.
		\]
		The claim follows.
	\end{proof}
	
	Setting $M = 8 \E \max_{i} |X_i| $ and $K = \max_{i}\| X_i \|_{\psi_2}$ consider the truncation scheme that is used in \eqref{Y_W_truncation}. Due to the assumption that all \( X_{i} \) are symmetrically distributed, we have \( \E Y_{i} = 0 \). Therefore, Lemma \ref{norm_AX_by_Gauss} implies
	\[
	\P\left(\sup_{A \in \mathcal{A}}\|A Y\| > C K(\E \sup_{A \in \mathcal{A}} \| A \gv \| + \sup_{A\in\AA}\| A \| \sqrt{t})  \right) \leq e^{-t},
	\]
	which can be used instead of the convex concentration inequality \eqref{sup_norm_A_X:eq} when dealing with the modified logarithmic Sobolev inequality. Following this proof and using the fact that \( \max_i|Y_i| \leq M \) almost surely, we end up with the concentration bound
	\[
	Z_{\AA}(Y) - \E Z_{\AA}(Y) \lesssim M K \left(\E \sup_{A \in \mathcal{A}} \| A \gv \| \sqrt{t} + \sup_{A\in\AA}\| A \| t \right),
	\]
	which holds with probability at least \( 1- e^{-t} \) for any \( t > 1 \). Furthermore, we slightly modify the derivations of the previous section by using Lemma~\ref{norm_AX_by_Gauss} instead of \eqref{norm_AX_conc_subgauss}. In particular, we get with probability at least \( 1- e^{-t} \) for any \( t > 1 \),  
	\[
	|Z_{\AA}(X) - Z_{\AA}(Y)|  \lesssim  M K (\E \sup\limits_{A\in\AA} \| A G \| \sqrt{t} +  \sup\limits_{A\in\AA} \| A \| t ),
	\]
	and taking expectation we also get
	\(
	|\E Z_{\AA}(X) - \E_{\AA} Z(Y)| \lesssim  M K \E \sup\limits_{A\in\AA} \| A G \|
	\).
	The claim follows from \eqref{ZY_ZX_comp}.

	\section{The matrix Bernstein inequality in the subexponential case}
	\label{matrixbern} 
	As we mentioned above, one of the prominent applications of the uniform Hanson-Wright inequalities is {a} recent concentration result in the Gaussian covariance estimation problem. It is known that covariance estimation may be alternatively approached by the matrix Bernstein inequality, {see e.g. \cite{wei2017estimation, Lounici14}}. Following the truncation approach, which was taken above, we provide a version of matrix Bernstein inequality that does not require uniformly bounded matrices. The standard version of the inequality (see \cite{Tropp12} and reference therein) 
	may be formulated as follows: consider random independent matrices \( X_{1}, \dots, X_{N} \in \R^{n \times n}\), such that almost surely $\max_i\|X_i\| \le L$. It holds
	\[
	\P\left( \left\| \sum_{i = 1}^{N} X_{i} - \E X_{i} \right\| > u \right) \leq n\exp\left( -c  \left(\frac{u^{2}}{\sigma^2} \bigwedge \frac{u}{L} \right)\right),
	\]
	where $c$ is an absolute constant and $\sigma^2 = \bigl\|\E\sum_{i = 1}^N (X_{i} - \E X_i)^2 \bigr\|$. The first problem with this result is that it does not hold in general cases when $\max_i\|X_i\|_{\psi_1}$ or $\max_i\|X_i\|_{\psi_2}$ are bounded. The second problem is the bound depends on the dimension $n$. This does not allow to apply this result to operators in infinite-dimensional Hilbert spaces. 
	
	For a positive-definite real square matrix $A$ we define the \emph{effective rank} as $\tilde{\mathbf{r}}(A) = \frac{\tr(A)}{\|A\|}$. We show the following bound. 
	\begin{proposition}
		\label{bernbound}
		Consider the set of random independent {Hermitian} matrices \( X_{1}, \dots, X_{N} \in {\C^{n \times n}} \) such that 
		$
		\bigl\| \| X_{i} \| \bigr\|_{\psi_{1}} < \infty.
		$
		Set $M = \bigl\| \max_{i \leq N} \| X_i \| \bigr\|_{\psi_1}$ and let {the positive definite Hermitian} matrix $ R $ be such that $\E\sum_{i = 1}^N X_{i}^2 \preceq R $. Finally, set $\sigma^2 = \|R\|$. There are absolute constants $c, C, c_1 > 0$ such that for any 
		$u \ge c_1\max\{M,\sigma\}$ we have
		\[
		\P\left( \left\| \sum_{i = 1}^{N} X_{i} - \E X_{i} \right\| > u \right) \leq C\tilde{\mathbf{r}}\left(R\right) \exp\left( -c  \left(\frac{u^{2}}{\sigma^2} \bigwedge \frac{u}{M} \right)\right).
		\]
	\end{proposition}
	\begin{remark}
		Using the well known bound for the maximum of subexponential random variables (see \cite{ledoux2013probability}) we have 
		\[
		\bigl\| \max_{i \leq N} \| X_i \| \bigr\|_{\psi_1} \lesssim \log N \  \max_{i \leq N}\bigl\| \| X_i \| \bigr\|_{\psi_1}.
		\]
		Therefore, up to absolute constant factors we may state the bound for $M = \log N\  \max_{i \leq N}\bigl\| \| X_i \| \bigr\|_{\psi_1}$.
		When $n = 1$ the effective rank plays no role and our bound recovers the version of classical Bernstein inequality which is due to \cite{Adam08}. In this paper, it is also shown that the $\log N$ factor cannot be removed in general. This means that $M = \bigl\| \max_{i \leq N} \| X_i \| \bigr\|_{\psi_1}$ can not be replaced by $\max_{i \leq N}\bigl\| \| X_i \| \bigr\|_{\psi_1}$ in general. 
	\end{remark}
	
	\begin{proof}
		Fix \( U > 0 \) and consider the decomposition 
		\[
		X_{i} = Y_{i} + Z_{i},
		\qquad
		Y_{i} = X_{i} \Ind( \| X_i\| \leq U ),
		\qquad
		Z_{i} = X_{i} \Ind( \| X_i\| > U ),
		\]
		so that the matrices \( Y_{i} \) are uniformly bounded by \( U \) in the operator norm. By the triangle inequality and the union bound we have
		\[
		\P\left( \left\| \sum_{i = 1}^{N} X_i - \E X_i \right\| > 2u \right) \leq
		\P\left( \left\| \sum_{i = 1}^{N} Y_i - \E Y_i \right\| > u \right)
		+
		\P\left( \left\| \sum_{i = 1}^{N} Z_i - \E Z_i \right\| > u \right).
		\]
		Therefore, two parts can be treated separately. Throughout this proof $c > 0$ is an absolute constant which may change from line to line. It is known that uniformly bounded random matrices satisfy the Bernstein-type inequality (see Theorem 3.1 in \cite{Minsk17}) for $u \ge \frac{1}{6}(U + \sqrt{U^2 + 36\sigma^2})$
		\[
		\P\left( \left\| \sum_{i = 1}^{N} Y_i - \E Y_i \right\| > u  \right) \leq 14\tilde{\mathbf{r}}\left(\E\sum_{i = 1}^N (Y_{i} - \E Y_i)^2\right)\exp\left( -\frac{c u^{2}}{\left\|\sum\limits_{i = 1}^N {\E}(Y_i - \E Y_i)^2\right\| + Uu}  \right),
		\]
		where we used $ \| Y_i\| \leq U $. However, since we want to present this bound in terms of $X_i$ and not $Y_i$, we need the following modification of the proof of Minsker's theorem. Using the notation of his proof, it follows from Lemma 3.1 in \cite{Minsk17}:
		\[
		\log\E\exp(\theta(Y_i - \E Y_i))  \preceq  \frac{\phi(\theta U)}{U^2}\E (Y_i - \E Y_i)^2 \preceq \frac{\phi(\theta U)}{U^2} 2 \E Y_i^2 \preceq \frac{\phi(\theta U)}{U^2} 2\E X_i^2 ,
		\]
		{where \( \phi(t) = e^t - t - 1\).} Now, using the same lines of the proof, instead of formula (3.4) we have
		\[
		\E \tr\phi\left(\theta\sum\limits_{i = 1}^N(Y_i - \E Y_i)\right) \le \tr\left(\exp\left(\frac{\phi(\theta U)}{U^2} 2 \sum_{i = 1}^{N} \E X_i^2 \right)- I_d\right),
		\]
		and lines (3.5) with the condition \( \sum_{i = 1}^{n} \E X_i^2 \preceq R \) imply
		\[
		\exp\left(\frac{\phi(\theta U)}{U^2} 2 \sum_{i = 1}^{N} \E X_i^2 \right) - I_d \preceq \exp\left(\frac{2 \phi(\theta U)}{U^2} R \right) - I_d \preceq \frac{R}{\sigma^2}\exp\left(\frac{2\phi(\theta U)}{U^2}\sigma^2\right),
		\]
		where $\sigma^2 = \left\|R\right\|$. Following the last lines of the proof of Theorem 3.1, we finally have
		\begin{equation}
		\label{bernst}
		\P\left( \left\| \sum_{i = 1}^{N} Y_i - \E Y_i \right\| > u  \right) \leq 14\tilde{\mathbf{r}}\left(R\right)\exp\left( -\frac{c u^{2}}{\sigma^2 + Uu}  \right), 
		\end{equation}
		for $u \ge C\max\{U,\sigma\}$. 
		
		We proceed with the analysis of \( Z_i \). Set \( U = 8\E\max\limits_{i \leq N} \| X_i \| \), then we have by Markov's inequality
		\[
		\P\left(  \max_{k \leq N}\left\| \sum_{i = 1}^{k} Z_{i} \right\| > 0 \right)  \leq  \P\left( \max_{i \leq N} \| Z_{i} \| > 0 \right)	
		=
		\P\left( \max_{i\leq N} \| X_{i} \| > U \right) \leq 1/8.
		\]
		Thus, we can apply Proposition 6.8 from \cite{ledoux2013probability} to \( Z_i \) taking its values {in} the Banach space \( ({\C^{n \times n}}, \| \cdot \|) \) equipped with the spectral norm. We have
		\[
		\E \left\| \sum_{i = 1}^{N} Z_{i} \right\| \leq 8 \E \max_{i \leq N} \| Z_{i} \| ,
		\]
		which implies that for some absolute constant \( K > 0 \),
		\[
		\E \left\| \sum_{i = 1}^{N} Z_{i} - \E Z_i \right\| \leq 2 \E \left\| \sum_{i = 1}^{N} Z_{i} \right\| \leq 16 \E \max_{i \leq N} \| Z_{i} \|
		\leq K \bigl\| \max_{i \leq N} \| Z_{i} \| \bigr\|_{\psi_1}.
		\]
		Using Theorem 6.21 from \cite{ledoux2013probability} in \( ({\C^{n \times n}}, \| \cdot \|) \) we have,
		\begin{align*}
		\left\| \left\| \sum_{i = 1}^{N} Z_{i} - \E Z_i \right\| \right\|_{\psi_1} & \leq  K_1 \left( \E \left\| \sum_{i = 1}^{N} Z_{i} - \E Z_i \right\| + \bigl\| \max_{i \leq N} \| Z_i \| \bigr\|_{\psi_1} \right) \\
		& \leq 
		K_2 \bigl\| \max_{i \leq N} \| Z_i \| \bigr\|_{\psi_1},
		\end{align*}
		where \( K_1, K_2 > 0 \) are absolute constants.
		This implies that for any $u \ge \bigl\| \max_{i \leq N} \| Z_i \| \bigr\|_{\psi_1}$ we have
		\[
		\P\left( \left\| \sum_{i = 1}^{N} Z_{i} - \E Z_i \right\| > u  \right) \leq \exp \left(  - \frac{c u}{\bigl\| \max_{i \leq N} \| Z_i \| \bigr\|_{\psi_1}} \right),
		\]
		where \( c > 0 \) is an absolute constant. Combining it with \eqref{bernst} and that for some absolute $C > 0$ we have \( U \le C\bigl\| \max_{i \leq N} \| X_i \| \bigr\|_{\psi_1} \) and $\bigl\| \max_{i \leq N} \| Z_i \| \bigr\|_{\psi_1} \le \bigl\| \max_{i \leq N} \| X_i \| \bigr\|_{\psi_1}$, we prove the claim.
	\end{proof}
	
	To the best of our knowledge, the Proposition \ref{bernbound} is the first to combine two important properties: it simultaneously captures the effective rank instead of the dimension $n$ and is valid for matrices with subexponential operator norm ({the} matrix Bernstein inequality in the unbounded case was {previously} granted under the so-called Bernstein moment condition; we refer to \cite{Tropp12} and the references therein). We should also compare our results with Proposition 2 of \cite{Koltch11}. His inequality has the same form as our bound, but instead of the effective rank, the original dimension $n$ is used and $M = \bigl\| \max_{i \leq N} \| X_i \| \bigr\|_{\psi_1}$ is replaced by $\max_{i \leq N}\bigl\|\| X_i \| \bigr\|_{\psi_1}\log\left(N\left(\max_{i \leq N}\bigl\|\| X_i \| \bigr\|_{\psi_1}\right)^2/\sigma^2\right)$.

	\subsubsection*{An application to covariance estimation with missing observations}
	
	Now we turn to the problem studied in \cite{Koltch17} and \cite{Lounici14}. Suppose, we want to estimate the covariance structure of a random subgaussian vector $X \in \mathbb{R}^n$ (which will be assumed centered) based on $N$ i.i.d. observations $X_1, \ldots, X_N$. For the sake of brevity, we work with the finite-dimensional case, while as in \cite{Koltch17} our results do not depend explicitly on the dimension $n$. Recall that a centered random vector $X \in \mathbb{R}^n$ is \emph{subgaussian} if for all $u \in \mathbb{R}^n$ we have
	\begin{equation}
	\label{subg}
	\|\langle X, u\rangle\|_{\psi_2} \lesssim (\E\langle X, u\rangle^2)^{\frac{1}{2}}.
	\end{equation}
	Observe that this definition does not require any independence of {the} components of $X$.
	
	In what follows we discuss a more general framework suggested by \cite{Lounici14}. Let $\delta_{i, j}$, $i \le N, j \le n$ be independent Bernoulli random variables with the same mean $\delta$. We assume that instead of observing $X_{1}, \ldots, X_N$ we observe vectors $Y_{1}, \ldots, Y_N$ which are defined as $Y_{i}^{j} = \delta_{i,j}X_i^{j}$. This means that some components of vectors $X_{1}, \ldots, X_N$ are missing (replaced by zero), each with probability $1 - \delta$. Since $\delta$ can be easily estimated, we assume it to be known. Following \cite{Lounici14}, denote 
	\[
	\hat{\Sigma}^{(\delta)} = \frac{1}{N}\sum_{i = 1}^{N} Y_iY_i^\T. 
	\]
	It can be easily shown that the estimator 
	\[
	\hat{\Sigma} = (\delta^{-1} - \delta^{-2})\Diag(\hat{\Sigma}^{(\delta)}) + \delta^{-2} \hat{\Sigma}^{(\delta)}
	\]
	is an unbiased estimator of $\Sigma = \E X_i X_i^\T$. In particular,
	\begin{equation}
	\label{Sigma}
	\Sigma = (\delta^{-1} - \delta^{-2})\Diag(\E Y_iY_i^\T) + \delta^{-2}\E Y_iY_i^\T .
	\end{equation}
	\begin{theorem}\label{missing:prop}
		Under the assumptions defined above, it holds with probability at least $1 - e^{-t}$ for \( t \geq 1 \) that
		\[
		\| \hat{\Sigma} - \Sigma \| \lesssim \| \Sigma \| \max \left( \sqrt{\frac{\tilde{\mathbf{r}}(\Sigma) \log \tilde{\mathbf{r}}(\Sigma)}{N\delta^{2}}}, \sqrt{\frac{t}{N\delta^{2}}} , \frac{\tilde{\mathbf{r}}(\Sigma)(\log\tilde{\mathbf{r}}(\Sigma) + t)\log N }{N \delta^2}   \right).
		\]
	\end{theorem}
	\begin{remark}
		The {upper bound} above provides some important improvements upon Proposition 3 in \cite{Lounici14}, which is
		\begin{equation}
		\label{oldresult}
		\| \hat{\Sigma} - \Sigma \| \lesssim \| \Sigma \| \max \left( \sqrt{\frac{\tilde{\mathbf{r}}(\Sigma) \log n}{N\delta^{2}}}, \sqrt{\frac{\tilde{\mathbf{r}}(\Sigma)t}{N\delta^{2}}} , \frac{\tilde{\mathbf{r}}(\Sigma)(\log n + t)(\log N + t)}{N \delta^2}   \right)
		\end{equation}
		The bound \eqref{oldresult} depends on $n$ and therefore is not applicable in the infinite dimensional scenarios. It also contains a term proportional to \( t^{2} \), which appeared due to a straightforward truncation of each observation. Moreover, this result has an unnecessary factor $\tilde{\mathbf{r}}(\Sigma)$ in the term $\sqrt{\frac{\tilde{\mathbf{r}}(\Sigma)t}{N\delta^{2}}}$. Finally, when $\delta = 1$ tighter results may be obtained using high probability generic chaining bounds for quadratic processes. In particular, Theorem 9 in \cite{Koltch17} implies {with probability at least \( 1 - e^{-t} \),}
		\begin{equation}
		\label{Kotcheqnew}
		\|\hat{\Sigma} - \Sigma\| \lesssim \|\Sigma\|\max\left(\sqrt{\frac{\tilde{\mathbf{r}}(\Sigma) }{N}}, \sqrt{\frac{t}{N}}, \frac{\tilde{\mathbf{r}}(\Sigma)}{N},  \frac{t}{N} \right) \, .
		\end{equation}
		Unfortunately, this analysis may not be implied for $\delta < 1$ in general, since the assumption \eqref{subg} does not hold for the vector $Y$, defined by $Y_{i}^{j} = \delta_{i,j}X_i^{j}$. Therefore, our technique is a reasonable alternative {that} works for general $\delta$ and is almost as tight as \eqref{Kotcheqnew} when $\delta = 1$. {We also remark that for $\delta = 1$ even sharper versions of \eqref{Kotcheqnew} were obtained in \cite{Mendelson18}. However, their statistical procedure differs from the sample covariance matrix $\hat{\Sigma}$.}
	\end{remark}
	
	To prove Theorem \ref{missing:prop} we need the following technical lemma, parts of which may as well be found in \cite{Lounici14}.
	\begin{lemma}
		\label{trace}
		Let $X \in \mathbb{R}^n$ satisfy \eqref{subg} with covariance matrix $\Sigma${, and} $Y = (\delta_{1}X^1, \ldots, \delta_n X^n)$ where $\delta_i$, $i \le n$ are independent Bernoulli random variables with the same mean $\delta$. We have
		\[
		\bigl\|\|\Diag(YY^\T)  \|\bigr\|_{\psi_1}\lesssim \tilde{\mathbf{r}}(\Sigma)\|\Sigma\|, \qquad
		\bigl\|\|\Off(YY^\T)  \|\bigr\|_{\psi_1}  \lesssim \tilde{\mathbf{r}}(\Sigma)\|\Sigma\|.
		\]
		Additionally, it holds for some absolute constant $C > 0$ that
		\begin{equation}
		\label{Lounici}
		\E \Off(Y Y^{\T})^2 \preceq C\delta^2 \tr(\Sigma) (\Sigma + \Diag(\Sigma)), \quad \text{and} \quad \E \Diag(YY^\T)^2 \lesssim C\delta \tr(\Sigma) \Diag(\Sigma).
		\end{equation}
	\end{lemma}
	
	\begin{proof}
		Observe that \( \| \Diag(Y Y^{\T}) \| \leq \| Y \|^{2}  \) and \( \| \Off(YY^\T) \| \leq \| Y Y^{\T} \| + \| \Diag(Y Y^{\T}) \| \leq 2 \| Y \|^{2} \). Therefore,
		\[
		\left\| \| \Off(YY^\T) \| \right\|_{\psi_1} \leq 2 \| \| Y \| \|_{\psi_2}^{2} \leq 2 \| \| X \| \|_{\psi_2}^{2} \lesssim \tr(\Sigma), 
		\]
		and the same bound holds for \(  \left\| \| \Diag(YY^\T) \| \right\|_{\psi_1}  \).
		
		Let \( A \) be an arbitrary symmetric matrix and let us calculate \( \E (A \odot \deltav \deltav^{\T})^2 \) where \( \odot \) denotes the Hadamard product and \( \deltav = (\delta_1, \dots, \delta_n) \) is a vector with independent components having Bernoulli distribution with the same mean $\delta$. We have,
		\[
		\left[\E (A \odot \deltav \deltav^{\T})^2\right]_{ii} = \E \sum_{k} A_{ik} \delta_i \delta_k  A_{ki} \delta_i \delta_k = \sum_{k} A_{ik} A_{ik} \E \delta_i^2 \delta_k^2 = \delta^2 [A^2]_{ii} + (\delta - \delta^2) A_{ii}^2.
		\]
		If $i \neq j$ we have for the \( i,j \)-th element 
		\begin{align*}
		\left[\E (A \odot \deltav \deltav^{\T})^2\right]_{ij} &= \E \sum_{k} A_{ik} \delta_i \delta_k A_{kj} \delta_j \delta_k = \sum_{k} A_{ik} A_{kj} \E \delta_i \delta_j \delta_k^2 \\ &= \delta^3 [A^2]_{ij} + (\delta^2 -\delta^3)(A_{ii} A_{ij} + A_{ij} A_{jj}).
		\end{align*}
		This can be put together in the following expression,
		\begin{align*}
		\E (\deltav \deltav^{\T} \odot A)^{2} &=  \delta^3 A^2 + (\delta^2 - \delta^3) \left[ \Diag(A^2) + \Off(A) \Diag(A) + \Diag(A) \Off(A)\right] \\
		& + (\delta - \delta^2) \Diag(A)^2.
		\end{align*}
		Note that all of these matrices are positive semi-definite, apart from the term \( \Off(A) \Diag(A) + \Diag(A) \Off(A) \), which we can obviously bound by \( \frac{1}{2} ( \Off(A) + \Diag(A)  )^2 = A^2 / 2 \). Taking into account \( \delta \leq 1 \), we have the following
		\begin{align*}
		\E (\deltav \deltav^{\T} \odot A)^{2} & \preceq  \frac{1}{2}(\delta^3 + \delta^2) A^2 + (\delta^2 - \delta^3) \Diag(A^2) + (\delta - \delta^2) \Diag(A)^2 \\
		&\preceq \delta^2( A^2 + \Diag(A^2)) + \delta \Diag(A)^2.
		\end{align*}
		Recall that \( Y = \diag(\deltav) X \). Therefore, we have \( \Off(YY^\T) = \deltav \deltav^{\T} \odot \Off(X X^\T) \). Since the latter has zero diagonal, the term with \( \delta \) in the formula above disappears. Therefore,
		\begin{equation}\label{YYT_to_XXT}
		\E \Off(Y Y^\T)^2 \preceq \delta^2 \left[ \E \Off(XX^\T)^2 + \Diag\left(  \E \Off(XX^\T)^2  \right) \right] .
		\end{equation}
		It holds \( \E \Off(XX^{{\T}})^2 \preceq 2 \E (XX^\T)^2 + 2 \E \Diag(XX^{{\T}})^2 \), and we also have from \cite{Lounici14} that \( \E (XX^\T)^2 \preceq C\tr(\Sigma) \Sigma \). Additionally, due to \eqref{subg} we immediately have \( \E X_i^4 \lesssim \Sigma_{ii}^2 \). Finally, the following bound holds
		\[
		\E \Diag(XX^\T)^2  \preceq C\Diag(\Sigma)^2 \preceq C\tr(\Sigma) \Diag(\Sigma).
		\]
		Plugging {these} bounds into \eqref{YYT_to_XXT} we get the second inequality. As for the diagonal case we have for \( A = \Diag(XX^\T) \),
		\[
		{\E \Diag(YY^\T)^2} \preceq {C} \delta \E \Diag(XX^\T)^2 \preceq C \delta \tr(\Sigma) \Diag(\Sigma).
		\]
	\end{proof}

	\begin{lemma}\label{norm}
		For \( Y \) as in Lemma~\ref{trace} and any unit {vector} \( u \in \R^{n} \) we have
		{\[
		\E ( u^{\T} \Off(YY^\T) u )^{{2}} \lesssim \delta^{2} \| \Sigma \|^2,
		\qquad
		\E ( u^{\T} \Diag(YY^{\T}) u )^{{2}} \lesssim \delta \| \Sigma \|^2.
		\]}
	\end{lemma}
	\begin{proof}
		Let \( v \in \R^{n} \) be an arbitrary unit vector. First, we want to check that
		\begin{equation}
		\label{uXXTv_L4}
		\| u^{\T} \Diag(XX^{\T}) v \|_{L_4} \lesssim \| \Sigma \|,
		\qquad
		\| u^{\T} \Off(XX^{\T}) v \|_{L_4} \lesssim \| \Sigma \|.
		\end{equation}
		Obviously, \( \| u^{\T} XX^\T v \|_{L_4} \leq \| u^\T X \|_{L_8} \| v^{\T} X \|_{L_8} \lesssim \| \Sigma \| \), so it is enough to check that only for the diagonal. Let us apply the symmetrization argument. Let \( \epsv= (\eps_1, \dots, \eps_{{n}})^{\T} \) denote independent Rademacher variables. Then,
		\[
		u^{\T} \Diag(XX^{\T}) v = \E_{\epsv} \epsv^{\T} \diag(u) X X^{\T} \diag(v) \epsv = \E_{\epsv} u_{\epsv} X X^{\T} v_{\epsv},
		\]
		where \( u_{\epsv} = (u_1 \eps_1, \dots, u_{{n}} \eps_{{n}})^{\T} \) {(and similarly for \(v_{\epsv}\))} and \( \E_{\epsv} \) denotes the conditional expectation with respect to \( \eps \) given $X$. Then, by Jensen's and H\"older's inequalities,
		\[
		\E \left(u^{\T} \Diag(XX^{\T}) {v} \right)^{4} \leq \E \left(u_{\epsv}^{\T} X X^{\T} {v_{\epsv}} \right)^{4} = \E_{\epsv} \E^{1/2}[(u_{\epsv}^{\T} X)^{8} \cond \epsv] \E^{1/2} [(v_{\epsv}^{\T} X)^{8} \cond \epsv] \lesssim \| \Sigma \|^{4},
		\]
		thus implying \eqref{uXXTv_L4}.
		
		\def\bv{\mathbf{b}}
		{
		Consider two vectors \( \av, \bv \in \R^{n} \). We show the following bound,
		\begin{equation}\label{missing_delta_square}
		    \E \left( \sum_{i \neq j} \delta_i \delta_j a_i b_j \right)^{2} \leq 18 \delta^{2} \| \av \|^2 \| \bv \|^2 + 2 \delta^{4} \left( \sum_{i} a_i \right)^{2} \left( \sum_{i} b_i \right)^{2} .
		\end{equation}
		First, using \( \E Z^{2} = \E (Z - \E Z)^{2} + (\E Z)^{2} \) and the fact that \( \E \delta_i = \delta \) we have,
		\[
		    \E \left( \sum_{i \neq j} \delta_i \delta_j a_i b_j \right)^{2} = \E \left( \sum_{i \neq j} (\delta_i - \delta) (\delta_j - \delta) a_i b_j \right)^{2} + \left( \sum_{i \neq j} \delta^{2} a_i b_j \right)^{2} \, .
		\]
		To the first term we apply the decoupling inequality (Theorem~6.1.1 in \cite{Vershynin2016HDP}). Namely, defining \( \delta_1', \dots, \delta_n' \) as independent copies of \( \delta_1, \dots, \delta_n\) we have,
		\begin{align*}
		    \E \left( \sum_{i \neq j} (\delta_i - \delta) (\delta_j - \delta) a_i b_j \right)^{2} &\leq 16 \E \left( \sum_{i \neq j} (\delta_i - \delta) (\delta_j' - \delta) a_i b_j \right)^{2} \\
		    & =
		    16 \sum_{i \neq j} \sum_{k \neq l} \E (\delta_i - \delta) (\delta_j' - \delta) (\delta_k - \delta) (\delta_l' - \delta) a_i b_j a_k b_l,
		\end{align*}
		where in the last sum only the terms with \( k = i \) and \( l = j \) do not vanish. Since \(\E (\delta_i - \delta)^{2} = \delta(1 - \delta) \), we have
		\[
		    \E \left( \sum_{i \neq j} (\delta_i - \delta) (\delta_j - \delta) a_i b_j \right)^{2} \leq 16 \sum_{i \neq j} a_i^{2} b_j^{2} \delta^{2} (1 - \delta)^2
		    \leq 16 \delta^{2} \| \av \|^{2} \| \bv \|^{2} .
		\]
		It remains to bound the second term. Using \( (x + y)^{2} \leq 2x^2 + 2y^2 \) together with the Cauchy-Schwarz inequality, we have 
		\begin{align*}
		    \left( \sum_{i \neq j} \delta^{2} a_i b_j \right)^{2} &\leq 2 \delta^{4} \left( \sum_{i} a_i b_i \right)^{2} + 2 \delta^{4} \left(\sum_{i, j} a_i b_j\right)^{2} \\
		    & \leq 2 \delta^{4} \| \av \|^{2} \| \bv\|^{2} + 2 \delta^{4} \left(\sum_{i} a_i \right)^{2} \left(\sum_{i} b_i\right)^{2}.
		\end{align*}
		Putting these two bounds together and using \( \delta \leq 1 \) we get the required inequality.}	{
		Since \( u^{\T} \Off(YY^\T) v = \deltav^{\T} \diag(u) \Off(XX^{\T}) \diag(v) \deltav \), we can apply \eqref{missing_delta_square} with \( \av = \diag(u) X \) and \( \bv = \diag(u) X \). This implies
		\begin{align*}
		    \E \left( u^{\T} \Off(YY^\T) v \right)^{2} &\lesssim \delta^{2} \E \| \diag(u) X \|^{2}  \| \diag(v) X \|^{2} + \delta^{4} \E (u^{\T} X)^{2} (v^{\T} X)^{2}  \\
		    & \lesssim \delta^{2} \E^{1/2} \| \diag(u) X \|^{4} \E^{1/2} \| \diag(v) X \|^{4} + \delta^{4} \E^{1/2} (u^{\T} X)^{4} \E^{1/2} (v^{\T} X)^{4} \, .
		\end{align*}
		Due to \eqref{subg} we have \( \E^{1/4} (u^{\T} X)^{4} \lesssim \| \Sigma \|^{1/2} \). Moreover, the vector \( \diag(u) X \) also satisfies the subgaussian assumption \eqref{subg} and has the covariance matrix \( \diag(u) \Sigma \diag(u) \). Therefore, we have
		\[
		    \E^{1/2} \| \diag(u) X \|^{4} \lesssim \tr(\diag(u) \Sigma \diag(u)) \lesssim \sum_{i} u_{i}^{2} \Sigma_{ii} \lesssim \max_{i} \Sigma_{ii} \lesssim \| \Sigma \| ,
		\]
		where we used that $\|u\| = 1$. Similar inequalities hold for the vector \(v \).
        Therefore, we conclude that
		\[
		\E \left( u^{\T} \Off(YY^{\T}) u \right)^{2} \lesssim \delta^{2} \| \Sigma \|^{2}.
		\]
		Finally, we have for the diagonal term
		\begin{align*}
		\E \left( u^{\T} \Diag(YY^{\T}) v \right)^{2} & =  \E \left( \sum_{i} \delta_i u_i v_i X_i^{2} \right)^{2} = \delta^2 \E \left( u^\T \Diag(XX^{\T}) v \right)^2 + (\delta - \delta^{2}) \sum_{i} u_{i}^2 v_i^2 \E X_{i}^{4} \\
		& \lesssim  \delta^2 \| \Sigma \|^{2} + (\delta - \delta^{2})  \sum_{i} u_{i}^{2} v_i^2 \| \Sigma \|^2
		\lesssim \delta \| \Sigma \|^{2}.
		\end{align*}
		}
	\end{proof}%
	
	Before we present the proof of the deviation bound, let us recall the following version of Talagrand's concentration inequality for empirical processes. Remarkably, the following result can be proven using very similar techniques: at first, one may use the modified logarithmic Sobolev inequality to prove a version of Talagrand's concentration inequality in the bounded case and then use the same truncation argument as in the proof of Theorem~\ref{mainthm} to get the result in the unbounded case.
	
	\begin{theorem}[Theorem 4 in \cite{Adam08}]\label{Adam08:thm}
		Let \( X_1, \dots, X_N \in \mathcal{X} \) be a sample of independent random variables and \( \mathcal{F} \) be a countable class of measurable functions \( \mathcal{X} \mapsto \R \) such that \( \sup_{f \in \FF} \| f(X_i) \|_{\psi_1} < \infty \). Define
		\begin{equation}\label{adam_talagrand_Z}
		Z_{\FF} = \sup_{f \in \FF} \left| \sum_{i = 1}^{N} f(X_i) - \E f(X_i) \right|
		\end{equation}
		and \( \sigma^{2} = \sup_{f \in \FF} \sum_{i = 1}^{N} \E f^{2}(X_i) \). There is an absolute constant \( C > 0 \) such that
		\[
		\P\left(  Z_{\FF} > 2 \E Z_{\FF} + t \right) \leq \exp\left( -\frac{t^2}{4 \sigma^{2}} \right) + 3 \exp\left(  -\frac{t}{C \| \max_i \sup_f |f(X_i)|\|_{\psi_1} } \right).
		\]
	\end{theorem}
	
	\begin{proof}[Proof of Theorem~\ref{missing:prop}]
		At first, using \eqref{Sigma} we have
		\[
		\|\hat{\Sigma} - \Sigma\| \lesssim 
		\delta^{-1} \left\| \Diag(\hat{\Sigma}^{(\delta)}) - \E \Diag(\hat{\Sigma}^{(\delta)})\right\| + \delta^{-2} \left\| \Off(\hat{\Sigma}^{(\delta)}) - \E\Off(\hat{\Sigma}^{(\delta)})\right\|.
		\]
		Let us apply Proposition \ref{bernbound} to the term
	    $
		N\Off(\hat{\Sigma}^{(\delta)}) =  \sum_{i = 1}^{N} \Off(Y_{i} Y_{i}^{\T}),
		$
		where
		\[
		R = C N \delta^2 \tr(\Sigma) (\Sigma + \Diag(\Sigma)).
		\]
		We have \( \tilde{\mathbf{r}}(R) \leq 2 \tilde{\mathbf{r}}(\Sigma) \) and \( \| R \| \lesssim N \delta^2 \tr(\Sigma) \| \Sigma\|  \). Therefore, with probability at least $1 - e^{-t},$
		\begin{align}
		\| \Off(\hat{\Sigma}^{(\delta)}) - \E \Off(\hat{\Sigma}^{(\delta)}) \| & \lesssim 
		\max\left( \sqrt{\frac{ \delta^2 \tr(\Sigma) \| \Sigma\|(\log \rr(\Sigma) + t)}{N}}, \frac{\Tr(\Sigma)(\log \rr(\Sigma) + t)\log N}{N}  \right) \nonumber \\
		& = 
		\| \Sigma \| \max\left( \sqrt{\frac{\delta^2 \rr(\Sigma) (\log \rr(\Sigma) + t)}{N}}, \frac{\rr(\Sigma)(\log \rr(\Sigma) + t) \log N}{N}  \right).\label{missing_off_bound}
		\end{align}
		Integrating this bound (see e.g. Theorem~2.3 in \cite{boucheron2013concentration}) we easily get
		\[
		\E \| \Off(\hat{\Sigma}^{(\delta)}) - \E \Off(\hat{\Sigma}^{(\delta)}) \| \lesssim \| \Sigma \| \max\left( \sqrt{\frac{\delta^2 \rr(\Sigma) \log \rr(\Sigma)}{N}}, \frac{\rr(\Sigma) \log \rr(\Sigma) \log N}{N}  \right).
		\]
		Finally, we apply Theorem~\ref{Adam08:thm} to the set of functions $\mathcal F$ indexed by \( \gamma \in S^{n - 1} \) and defined by
		\[
		f_{\gamma}(X_i) = \gamma^{\T} \Off(Y_i Y_i^{\T}) \gamma,
		\]
		so that \( Z_{\mathcal F} = N \| \Off(\hat{\Sigma}^{(\delta)}) - \E \Off(\hat{\Sigma}^{(\delta)}) \|   \) in \eqref{adam_talagrand_Z}. Then, by Lemma~\ref{norm} we have \( \sigma^{2} \lesssim \delta^{2} N \| \Sigma \|^2 \) and by Lemma~\ref{trace} \( \| \max_{i} \sup_{f} |f(X_i)| \|_{\psi_1} = \| \max_{i} \| \Off(Y_i Y_i^{\T}) \| \|_{\psi_1} \lesssim \tilde{\mathbf{r}}(\Sigma) \| \Sigma \| \log N  \), so that with probability {at least} \( 1 - e^{-t} \) for \( t \geq 1, \)
		\begin{align*}
		\| \Off(\hat{\Sigma}^{(\delta)}) - \E \Off(\hat{\Sigma}^{(\delta)}) \| & \lesssim  \E \| \Off(\hat{\Sigma}^{(\delta)}) - \E \Off(\hat{\Sigma}^{(\delta)}) \| + \delta \| \Sigma \| \sqrt{ \frac{t}{N} } + \| \Sigma \| \frac{\tilde{\mathbf{r}}(\Sigma) t \log N}{N} \\
		& \lesssim 
		\| \Sigma \| \max\left( \sqrt{\frac{\delta^2\rr(\Sigma) \log \rr(\Sigma)}{N}}, \sqrt{\frac{\delta^{2} t}{N}}, \frac{\rr(\Sigma)(\log \rr(\Sigma) + t) \log N}{N}  \right).
		\end{align*}
		
		We proceed with the diagonal term. Applying Proposition~\ref{bernbound} to the sum \( N\Diag(\hat{\Sigma}^{(\delta)}) = \sum_{i = 1}^{N} \Diag(Y_{i} Y_{i}^{\T}) \) with \( R = CN \delta \tr(\Sigma) \Diag(\Sigma) \) we have \( \rr(R) \lesssim \rr(\Sigma) \) and \( \| R \| \lesssim N \delta \tr(\Sigma) \| \Sigma\| \). Therefore, with probability at least $1 - e^{-t}$ we have
		\begin{equation}
		\| \Diag(\hat{\Sigma}^{(\delta)}) - \E \Diag(\hat{\Sigma}^{(\delta)}) \|  \lesssim 
		\| \Sigma \| \max\left( \sqrt{\frac{ \delta \rr(\Sigma) (\log \rr(\Sigma) + t)}{N}}, \frac{\rr(\Sigma)(\log \rr(\Sigma) + t) \log N}{N}  \right). \label{missing_diag_bound}
		\end{equation}
		Again, integrating this inequality we get
		\[
		\E \| \Diag(\hat{\Sigma}^{(\delta)}) - \E \Diag(\hat{\Sigma}^{(\delta)}) \| \lesssim
		\| \Sigma \| \max\left( \sqrt{\frac{ \delta \rr(\Sigma) \log \rr(\Sigma) }{N}}, \frac{\rr(\Sigma) \log \rr(\Sigma) \log N}{N}  \right).
		\]
		We have  \( \E ( u^{\T} \Diag(Y_i Y_i^\T) u )^2 \lesssim \delta \| \Sigma \|^2 \) and  \(  \| \max_{i} \| \Off(Y_i Y_i^{\T}) \| \|_{\psi_1} \lesssim \tilde{\mathbf{r}}(\Sigma) \| \Sigma \| \log N \) by Lemma~\ref{norm} and Lemma~\ref{trace} respectively. By  Theorem~\ref{Adam08:thm} we have with probability at least \( 1 - e^{-t} \),
		\begin{align*}
		\| \Diag(\hat{\Sigma}^{(\delta)}) - \E \Diag(\hat{\Sigma}^{(\delta)}) \| & \lesssim  \E \| \Diag(\hat{\Sigma}^{(\delta)}) - \E \Diag(\hat{\Sigma}^{(\delta)}) \| +  \| \Sigma \| \sqrt{ \frac{\delta t}{N} } + \| \Sigma \| \frac{\tilde{\mathbf{r}}(\Sigma) t \log N}{N} \\
		& \lesssim 
		\| \Sigma \| \max\left( \sqrt{\frac{\delta\rr(\Sigma) \log \rr(\Sigma)}{N}}, \sqrt{\frac{\delta t}{N}}, \frac{\rr(\Sigma)(\log \rr(\Sigma) + t) \log N}{N}  \right).
		\end{align*}
		Finally, we combine the off-diagonal and diagonal bounds via the triangle  inequality and get
		\[
		\| \hat{\Sigma} - \Sigma \| \leq \delta^{-2} \| \Off(\hat{\Sigma}^{(\delta)}) - \E \Off(\hat{\Sigma}^{(\delta)}) \| + \delta^{-1} \| \Diag(\hat{\Sigma}^{(\delta)}) - \E \Diag(\hat{\Sigma}^{(\delta)}) \| \, .
		\]
	\end{proof}
	
	\section*{Acknowledgement}
	We are indebted to Rados\l{}aw Adamczak for his very useful feedback at several stages of this paper. We are especially grateful for his suggestion to study the question of Marton in the Ising model and for providing us with an important example in Section \ref{logarithm}.
	
	\bibliography{mybib}

\begin{thebibliography}{}

\bibitem[Adamczak, 2008]{Adam08}
Adamczak, R. (2008).
\newblock A tail inequality for suprema of unbounded empirical processes with
  applications to markov chains.
\newblock {\em Electronic Journal of Probability}, 13:1000--1034.

\bibitem[Adamczak, 2015]{Adam15}
Adamczak, R. (2015).
\newblock {A note on the Hanson-Wright inequality for random vectors with
  dependencies}.
\newblock {\em Electronic Communications in Probability}, 20.

\bibitem[Adamczak et~al., 2018a]{Adam18}
Adamczak, R., Kotowski, M., Polaczyk, B., and Strzelecki, M. (2018a).
\newblock {A note on concentration for polynomials in the Ising model}.
\newblock {\em arXiv:1809.03187}.

\bibitem[Adamczak et~al., 2018b]{adamczak2018hanson}
Adamczak, R., Lata{\l}a, R., and Meller, R. (2018b).
\newblock {Hanson-Wright inequality in Banach spaces}.
\newblock {\em arXiv:1811.00353}.

\bibitem[Adamczak and Wolff, 2015]{Adam15a}
Adamczak, R. and Wolff, P. (2015).
\newblock {Concentration inequalities for non-Lipschitz functions with bounded
  derivatives of higher order}.
\newblock {\em Probability Theory and Related Fields}, 162:531--586.

\bibitem[Arcones and Gine, 1993]{Arcones93}
Arcones, M. and Gine, E. (1993).
\newblock On decoupling, series expansions, and tail behavior of chaos
  processes.
\newblock {\em Journal of Theoretical Probability}, 6:101---122.

\bibitem[Berend and Kontorovich, 2013]{kontorovich13}
Berend, D. and Kontorovich, A. (2013).
\newblock On the concentration of the missing mass.
\newblock {\em Electron. Commun. Probab.}, 18,(3):7 pp.

\bibitem[Borell, 1984]{Borell84}
Borell, C. (1984).
\newblock {On the Taylor series of a Wiener polynomial. Seminar Notes on
  multiple stochastic integration, polynomial chaos and their integration}.
\newblock {\em Case Western Reserve Univ., Cleveland}.

\bibitem[Boucheron et~al., 2005]{Boucheron05}
Boucheron, S., Bousquet, O., Lugosi, G., and Massart, P. (2005).
\newblock Moment inequalities for functions of independent random variables.
\newblock {\em The Annals of Probability}, 33(2):514--560.

\bibitem[Boucheron et~al., 2003]{Boucheron03}
Boucheron, S., Lugosi, G., and Massart, P. (2003).
\newblock Concentration inequalities using the entropy method.
\newblock {\em The Annals of Probability}, 31(3):1583--1614.

\bibitem[Boucheron et~al., 2013]{boucheron2013concentration}
Boucheron, S., Lugosi, G., and Massart, P. (2013).
\newblock {\em Concentration inequalities: A nonasymptotic theory of
  independence}.
\newblock Oxford university press.

\bibitem[Dicker and Erdogdu, 2017]{Dicker17}
Dicker, L.~H. and Erdogdu, M. (2017).
\newblock Flexible results for quadratic forms with applications to variance
  components estimation.
\newblock {\em The Annals of Statistics}, 45(1):386--414.

\bibitem[G\"otze et~al., 2018]{Gotze18}
G\"otze, F., Sambale, H., and Sinulis, A. (2018).
\newblock Higher order concentration for functions of weakly dependent random
  variables.
\newblock {\em arXiv:1801.06348}.

\bibitem[Hitczenko et~al., 1998]{Hitczenko98}
Hitczenko, P., Kwapien, S., Li, W., Schechtman, G., Schlumprecht, T., and Zinn,
  J. (1998).
\newblock {Hypercontractivity and Comparison of Moments of Iterated Maxima and
  Minima of Independent Random Variables}.
\newblock {\em Electronic Journal of Probability}, 3.

\bibitem[Hsu et~al., 2012]{Hsu12}
Hsu, D., Kakade, S., and Zhang, T. (2012).
\newblock A tail inequality for quadratic forms of subgaussian random vectors.
\newblock {\em Electron. Commun. Probab.}, 17(52):6 pp.

\bibitem[Koltchinskii, 2011]{Koltch11}
Koltchinskii, V. (2011).
\newblock {Von Neumann entropy penalization and low-rank matrix estimation}.
\newblock {\em Annals of Statistics}, 39(6):2936--2973.

\bibitem[Koltchinskii and Lounici, 2017]{Koltch17}
Koltchinskii, V. and Lounici, K. (2017).
\newblock Concentration inequalities and moment bounds for sample covariance
  operators.
\newblock {\em Bernoulli}, 23(1):110--133.

\bibitem[Krahmer et~al., 2014]{Mendelson14}
Krahmer, F., Mendelson, S., and Rauhut., H. (2014).
\newblock {Suprema of Chaos Processes and the Restricted Isometry Property}.
\newblock {\em Communications in Pure and Applied Mathematics}, 67:1877--1904.

\bibitem[Ledoux, 2001]{Ledoux2001}
Ledoux, M. (2001).
\newblock {\em The concentration of measure phenomenon}, volume~89 of {\em
  Mathematical surveys and Monographs}.
\newblock American Mathematical Society.

\bibitem[Ledoux and Talagrand, 2013]{ledoux2013probability}
Ledoux, M. and Talagrand, M. (2013).
\newblock {\em Probability in Banach Spaces: Isoperimetry and Processes}.
\newblock Springer-Verlag Berlin Heidelberg.

\bibitem[Lounici, 2014]{Lounici14}
Lounici, K. (2014).
\newblock High-dimensional covariance matrix estimation with missing
  observations.
\newblock {\em Bernoulli}, 20(3):1029--1058.

\bibitem[Marton, 2003]{Marton03}
Marton, K. (2003).
\newblock Measure concentration and strong mixing.
\newblock {\em Studia scientiarum mathematicarum hungarica}, 40:95--113.

\bibitem[Mendelson and Zhivotovskiy, 2018]{Mendelson18}
Mendelson, S. and Zhivotovskiy, N. (2018).
\newblock Robust covariance estimation under ${L}_4-{L}_2$ norm equivalence.
\newblock {\em To appear in Annals of Statistics}.

\bibitem[Minsker, 2017]{Minsk17}
Minsker, S. (2017).
\newblock {On Some Extensions of Bernstein's Inequality for Self-adjoint
  Operators}.
\newblock {\em Statistics and Probability Letters}, 127:111--119.

\bibitem[Rudelson and Vershynin, 2013]{Rudelson13}
Rudelson, M. and Vershynin, R. (2013).
\newblock {Hanson-Wright inequality and sub-gaussian concentration}.
\newblock {\em Electron. Commun. Probab.}, 18.

\bibitem[Schlemm, 2016]{schlemm16}
Schlemm, E. (2016).
\newblock {The Kearns--Saul Inequality for Bernoulli and Poisson-Binomial
  Distributions}.
\newblock {\em Journal of Theoretical Probability}, 29:48---62.

\bibitem[Szarek, 1976]{szarek1976best}
Szarek, S. (1976).
\newblock {On the best constants in the Khinchin inequality}.
\newblock {\em Studia Mathematica}, 58(2):197--208.

\bibitem[Talagrand, 1996]{Talagrand96}
Talagrand, M. (1996).
\newblock New concentration inequalities in product spaces.
\newblock {\em Inventiones mathematicae}, 126:505---563.

\bibitem[Talagrand, 2014]{Talagrand2014}
Talagrand, M. (2014).
\newblock {\em Upper and lower bounds for stochastic processes: modern methods
  and classical problems}, volume~60.
\newblock Springer Science \& Business Media.

\bibitem[Tropp, 2012]{Tropp12}
Tropp, J. (2012).
\newblock {User-Friendly Tail Bounds for Sums of Random Matrices}.
\newblock {\em Foundations of Computational Mathematics}, 12:389---434.

\bibitem[van Handel, 2016]{Vanhandel16}
van Handel, R. (2016).
\newblock {Probability in High Dimension}.
\newblock {\em Lecture Notes, Princeton University}.

\bibitem[Vershynin, 2016]{Vershynin2016HDP}
Vershynin, R. (2016).
\newblock {\em High-Dimensional Probability: An Introduction with
  Applications}, volume~47 of {\em Cambridge Series in Statistical and
  Probabilistic Mathematics}.
\newblock Cambridge University Press.

\bibitem[Wei and Minsker, 2017]{wei2017estimation}
Wei, X. and Minsker, S. (2017).
\newblock Estimation of the covariance structure of heavy-tailed distributions.
\newblock In {\em Advances in Neural Information Processing Systems}, pages
  2859--2868.

\end{thebibliography}
	
	\appendix
	
	\section{An approximation argument for non-smooth functions}\label{smooth_approx:section}
	
	In order to apply the logarithmic Sobolev assumption \eqref{log-sobol_assume} rigorously we need to take smooth approximations of the function 
	\[ 
	Z_{\AA}(X) = \sup_{A \in \AA} (X^{\T} A X - \E X^{\T} A X ) .
	\] 
	Notice that we have,
	\[
	|Z_{\AA}(X) - Z_{\AA}(Y)| \leq \| X - Y \| \left( \sup_{A \in \AA} \| A X \| + \sup_{A \in \AA} \| A Y \|  \right).
	\]
	The following simple lemma shows how to apply the logarithmic Sobolev inequality to non-differentiable functions.
	
	\begin{lemma}
		Suppose that \( X \) satisfies Assumption~{\ref{logSobolev}}.  Let \( f :\R^{n} \rightarrow \R \) be such that 
		\[
		|f(x) - f(y)| \leq |x - y| \max(L(x), L(y)),
		\]
		for some continuous non-negative function \( L(x) \). Then for some absolute constant \( C > 0 \) and any \( \lambda \in \R \) it holds
		\[
		\Ent(e^{\lambda f}) \leq C K^{2} \lambda^{2} \E L(x)^{2} e^{\lambda f}
		\]
	\end{lemma}
	\begin{proof}
		Set \( h(x) = x^{2}(1 - x)_{+}^{2} \) and consider the smoothing kernel supported on {the} unit ball defined by
		\[
		\phi(u) = \frac{1}{Z_h} h(\| u \|^{2}),
		\qquad
		Z_h = \int h(\| u \|^2) du = S_{n - 1} \int_{ 0}^{\infty} h(r^2) dr,
		\]
		where \( S_{n-1} \) is a surface area of the unit sphere in \( \R^{n} \).
		Note that since \( \phi \) is radial, \( \nabla \phi(u) = - \nabla \phi(-u) \) and also,
		\[
		\int \| u \| \| \nabla \phi(u) \| du = \frac{2 S_{n-1}}{Z_h} \int_{0}^{\infty} r^{2} |g'(r)| dr = \frac{ 2 \int_{0}^{\infty} r^{2} |h'(r)| dr }{ \int_{ 0}^{\infty} h(r^2) dr } = C_{h}.
		\]
		Setting \( \phi_{m}(u) = m^{-1} \phi(u / m) \) for $m \in \mathbb{N}$ we have \( \nabla \phi_{m}(u) = m^{-2} (\nabla \phi)( u /m ) \). Therefore, we have
		\[  
		\int \| u \| \| \nabla \phi_{m}(u) \| du = \int \left\| \frac{u}{m} \right\| \left\| (\nabla \phi)\left( \frac{u}{m}\right) \right\| d\frac{u}{m} = C_{h}.
		\] 
		
		Take \(F(x) = e^{\lambda f(x)/2} \) and let us consider a sequence of smooth approximations
		\(
		F_{m}(x) = \int \phi_{m}\left(x - u\right) F(u) du,
		\)
		so that \( F_{m}(x) \) tends to \( F \) pointwise. This is possible due to the fact that \( F \) is continuous. Moreover, we have due to the symmetry
		\begin{align*}
		\nabla F_{m}(x) &=  \int (\nabla \phi_{m})(x - u) F(u) du = 
		\int  (\nabla \phi_{m})(u) F(x - u) du \\
		& = 
		\frac{1}{2} \int (\nabla \phi_{m})(u)[F(x - u) - F(x + u)] du.
		\end{align*}
		Since \( \phi_{m}(u) \) vanishes for \( \| u \| \geq 1/m \), we have
		\begin{align*}
		\| \nabla F_{m}(x) \| &\leq \frac{1}{2} \sup_{\| u \| \leq m^{-1}} \frac{|F(x - u) - F(x + u)|}{\| u \|} \int \| u \| \| \nabla \phi_{m}(u) \| du 
		\\
		&= C_{h} \sup_{\| u \| \leq m^{-1}} \frac{|F(x - u) - F(x + u)|}{2 \| u \|}.
		\end{align*}
		It is easy to see that
		\[
		|F(x) - F(y)| = |e^{\lambda f(x)/2} - e^{\lambda f(y)/2}| \leq \frac{\lambda}{2} \| x - y\| \max(e^{\lambda f(x)/2}, e^{\lambda f(y)/2}) \max(L(x), L(y)),
		\]
		and therefore,
		\[
		\| \nabla F_{m}(x) \| \leq \frac{\lambda C_{h}}{2} \widetilde{F}_{m}(x) \times L_{m}(x),
		\]
		where we set \( L_{m}(x) = \sup_{y \, : \, \|x-y\| \leq m^{-1} } L(y) \) and \( \widetilde{F}_{m}(x) = \sup_{\| x - y \| \leq m^{-1}} e^{\lambda f(y) /2} \) that tend pointwise to \( L(x) \) and \( F(x) \), respectively, as \( m \rightarrow \infty\). Since each function \( f_{m} \) is smooth, we have by Assumption~\ref{logSobolev} that
		\[
		\Ent(F_{m}^{2}) \leq K^{2} \E \| \nabla F_{m}(x) \|^{2} \leq \frac{\lambda^{2} C_{h}^{2}}{4} K^{2} \E L_{m}^{2}(x) \widetilde{F}_{m}(x)^{2}.
		\]
		Taking the limit \( m \rightarrow \infty \) we prove the the required inequality.
	\end{proof}

\end{document}